\documentclass[10pt]{article}

\usepackage{amsmath,amsthm,amssymb,amsfonts}
\usepackage{graphicx}
\usepackage{verbatim}
\usepackage{mathrsfs}
\usepackage{fancyhdr}
\usepackage{latexsym}
\usepackage{graphicx}
\usepackage{dsfont}
\usepackage{color}
\usepackage[hidelinks]{hyperref}
\usepackage{bm}
\usepackage{drawmatrix}
\usepackage{enumitem}
\usepackage[percent]{overpic}
\usepackage{epstopdf}
\usepackage{subcaption}
\usepackage{booktabs}

\usepackage{floatrow}
\newfloatcommand{capbtabbox}{table}[][\FBwidth]

\newcommand{\1}{\mathfrak 1}

\newcommand{\R}{\mathbb R}
\newcommand{\N}{\mathbb N}
\newcommand{\Z}{\mathbb Z}
\newcommand{\C}{\mathbb C}
\newcommand{\E}{\mathbb E}

\newcommand{\X}{\mathcal X} 
\newcommand{\Y}{\mathcal Y} 
\newcommand{\B}{\mathcal B} 
\newcommand{\cC}{\mathcal C} 
\newcommand{\D}{\mathcal D} 
\newcommand{\F}{\mathcal F}

\newcommand{\K}{\mathfrak K} 

\newcommand{\bX}{\bar{X}}

\newcommand{\bx}{\bar{x}}
\newcommand{\by}{\bar{y}}
\newcommand{\bz}{\bar{z}}
\newcommand{\bO}{\bar{\Omega}}
\newcommand{\brho}{\bar{\rho}}
\newcommand{\bbeta}{\bar{\beta}}
\newcommand{\bu}{\bar{u}}

\def\txti{{\textnormal{i}}}

\DeclareMathOperator{\conj}{conj}

\theoremstyle{plain}
\newtheorem{theorem}{Theorem}[section]
\newtheorem{proposition}[theorem]{Proposition}
\newtheorem{lemma}[theorem]{Lemma}

\newtheorem{remark}[theorem]{Remark}
\newtheorem{definition}[theorem]{Definition}

\newtheorem{example}[theorem]{Example}

\title{A posteriori validation of generalized polynomial chaos expansions}

\author{Maxime Breden \thanks{CMAP, \'Ecole Polytechnique, route de Saclay, 91120 Palaiseau, France. \texttt{maxime.breden@polytechnique.edu}}}

\begin{document}

\maketitle

\begin{abstract}
Generalized polynomial chaos expansions are a powerful tool to study differential equations with random coefficients, allowing in particular to efficiently approximate random invariant sets associated to such equations. In this work, we use ideas from validated numerics in order to obtain rigorous a posteriori error estimates together with existence results about gPC expansions of random invariant sets. This approach also provides a new framework for conducting validated continuation, i.e. for rigorously computing isolated branches of solutions in parameter-dependent systems, which generalizes in a straightforward way to multi-parameter continuation. We illustrate the proposed methodology by rigorously computing random invariant periodic orbits in the Lorenz system, as well as branches and 2-dimensional manifolds of steady states of the Swift-Hohenberg equation.
\end{abstract}


\section{Introduction}
\label{sec:intro}

Most of the mathematical models that are used nowadays to try and describe the world we live in, or at least some very specific region or aspect of it, include some stochastic component. This randomness can have various sources: sometimes we do not fully know or understand the mechanisms underlying the phenomenon we are trying to describe, sometimes we need to account for the influence of events occurring at much smaller scales than that of the full system, for which we cannot afford to solve too accurately, and sometimes our model contains crucial parameters whose value can only be known up to some uncertainty level.

A common mathematical framework to study this last situation is the one of random differential equations, say a random ODE described by a nonlinear vector field $f$
\begin{align}
\label{eq:RDE_intro}
X' = f(X,p),
\end{align}
or more generally a random PDE, where $p$ denotes a parameter whose value is not known precisely, and is therefore represented by a random variable. In this work we assume that the probability distribution of $p$ is known, for instance through some preliminary statistical inference. In this situation, one would like to understand and quantify as precisely as possible how the uncertainty in $p$ affects the output of the system~\cite{Sul15}.

If we want to study the global behavior of~\eqref{eq:RDE_intro}, one option is to use a Monte-Carlo type approach: sample $p$ according to its known distribution, and study for each sampled value $p_i$ the deterministic system $X' = f(X,p_i)$. Of course, even in the deterministic case, understanding the global dynamics of a system of nonlinear ODEs can already be a daunting task. Numerical simulations can then be of great help to get some insights, in particular in order to study invariant sets (equilibria, periodic orbits, invariant manifolds, connecting orbits, etc), which typically act as building blocks of the global dynamics.

Another option to study problems with random parameters like~\eqref{eq:RDE_intro}, which has risen in popularity in the last decades, is the usage of generalized polynomial chaos (gPC) expansions~\cite{GhaSpa91,Wie38}. The main idea is to expand the random quantity of interest as a series with a well chosen basis, namely polynomials in $p$ which are orthogonal with respect to probability distribution of $p$. One is then left with computing the (deterministic!) coefficients of this series expansion, and as in the Monte-Carlo approach we recover a deterministic problem, and the ability to use existing algorithms for it. This strategy has proven very effective in various contexts~\cite{LeMKni10,Xiu09}, and in particular the recent work~\cite{BreKue20} showcases that gPC can be used to approximate some random invariant sets generated by random ODEs of the form~\eqref{eq:RDE_intro}.

Once the gPC expansion has been computed, it readily provides quantitative information about the way the randomness in $p$ influences the solutions of the system. In order to be more concrete, let us focus for instance on periodic orbits. With a gPC representation, we directly have access to the mean and the variance of the period (which typically depends in a non-explicit, nonlinear way on $p$), and we can also do cheaper Monte-Carlo simulations to estimate the full probability distribution of the period, or to quantify the shape of the orbit in phase space, etc.

\medskip

The above discussion exemplifies why gPC is a very powerful tool to quantify uncertainties, at least if we had access to the \emph{exact} gPC representation of the object of interest. However, in practice, the fact that we heavily rely on numerical computations introduces an extra level of uncertainty. The two main sources of approximations in the above procedure are: the fact that the (theoretically infinite) gPC series expansion is truncated, because we can only compute finitely many coefficients, and the fact that the deterministic algorithms used to compute these gPC coefficients also contain truncation errors (indeed even for a fully deterministic nonlinear ODE, one cannot hope to compute exactly periodic orbits, or more complicated invariant sets).

Regarding the truncation of the gPC expansion, it is known a priori that the truncation error decays quickly (spectral convergence) when $X$ depends smoothly on $p$~\cite{CanHusQuaTho12,Fun08}, and some tight convergence results were even obtained recently in a non-smooth case~\cite{BouGobRey21}. However, when $X$ is not known a priori, these estimates cannot give any quantitative information about the truncation error. A posteriori error estimators for gPC expansions have also been developed, especially in the context of random linear elliptic PDEs~\cite{DebBabOde01,EigGitSchZan14,BesPowSil14}, but also for more general random PDEs~\cite{ButDawWid11,MatLeM07,MeyRohGie20}. Yet again, for nonlinear problems these estimators typically still contain some approximations and cannot provide fully rigorous error bounds between the approximate solution and the exact one, if only because the existence of an exact solution is not always readily available.

\medskip

The purpose of this work is to quantify in a very explicit way all the errors involved in the computation of some random invariant sets using gPC. For instance, if $\bX$ is a gPC representation of an approximate random periodic solution that we obtained numerically, we are going to provide guaranteed a posteriori estimates stating that there exists an exact random periodic solution $X^*$, with $\Vert \bX - X^*\Vert\leq r$ in some well chosen norm, where the error bound $r$ will be explicit. In the context of deterministic dynamical systems, such guaranteed a posteriori error estimates which also provide existence results go back at least to the proof of the Feigenbaum conjecture~\cite{EckWit87,Lan82} (see also~\cite{Ura65} for an even earlier work) and have become more and more popular since then, mostly under the name of \emph{validated/rigorous numerics} or \emph{computer-assisted proofs}. We will recall some of the main ideas behind these techniques in this work, and refer to  the survey papers~\cite{Gom19,KapMroWilZgl21,KocSchWit96,Rum10,BerLes15} and books~\cite{NakPluWat19,Tuc11} for a more in-depth overview of the field. The main contribution of this work is to show that these ideas can be extended, in a computationally efficient way, to dynamical systems with random coefficients.

\medskip

Before proceeding further, let us present an alternate viewpoint for the techniques we develop in this paper. The important starting observation is that, in any kind of gPC expansion, the probability distribution of $p$ only influences the choice of the expansion basis. Once a basis has been selected, one can forget the random character of $p$, and simply view~\eqref{eq:RDE_intro} as a deterministic parameter-dependent problem. In that context, numerical continuation methods can be used, for instance to approximate a curve of periodic orbits. This is exactly what we do with a gPC expansion, the only difference being that traditional continuation methods would typically proceed by computing points close to one another along the curve, and glue them together in a low-order (say piece-wise linear) fashion, whereas here we directly compute a larger chunk of curve at once, by looking for a higher order parameterization. 

Numerical continuation can be used together with validated numerics to prove the existence of curves of solutions and to get tight and explicit error bounds (see e.g. \cite{AriKoc10,BreLesVan13,BerLesMis10,BerQue21,Wan18}), but up to now this has mostly been done with the piece-wise linear approximations provided by usual predictor-corrector techniques. The only exceptions seems to be the recent works~\cite{Ari22,AriGazKoc21}, where Taylor expansions in the parameter are used to compute and validate larger pieces of curve at once. The approach proposed in this paper is very similar, but we generalize it to other kind of expansions bases, which proves to be sometimes more efficient than using Taylor expansions. This framework also generalizes in a completely straightforward way to rigorous mutli-parameter continuation, which again provides a higher-order and more global alternative to the existing techniques~\cite{GamLesPug16}, which also rely on local piece-wise linear approximations.

\medskip

The remainder of the paper is organized as follows. In Section~\ref{sec:background}, we introduce some of the tools that will be required in this work, in particular well chosen sequence spaces provided with a discrete convolution and a type of Newton-Kantorovich Theorem, and start with a basic example (Section~\ref{sec:BasicExample}) describing how these tools can be combined to rigorously validate gPC expansions. We then explain in Section~\ref{sec:Lorenz} how this framework can be applied to random invariant sets, via the example of random periodic orbits in the Lorenz system. This section ends with some comparisons regarding the performance of several choices of polynomial bases. We continue with a different example in Section~\ref{sec:SH}, namely the Swift-Hohenberg equation, for which we rigorously compute parameter-dependent families of steady states. With this example we focus more on the validation continuation viewpoint, and on how the proposed technique interacts with bifurcations, and also discuss how to handle multiple parameters at once. We wrap up in Section~\ref{sec:conclusion}, where we summarize our work, and discuss the current limitations and possible extensions of the proposed approach. All the codes associated with this work are available at~\cite{BreGit}.

\section{Background material, notations and a basic example}
\label{sec:background}

In this section, we introduce some of the objects and tools that we make use of in this work. Most of the material presented here is not original, and mainly included for the convenience of the reader, and for the sake of fixing some notations. We discuss weighted $\ell^1$ spaces of Fourier coefficients in Section~\ref{sec:Fourier}, and an extension where each Fourier coefficient is itself written as a gPC expansions together with associated generalized convolutions structures in Section~\ref{sec:gene_convo_prod}. We introduce notations for finite dimensional projections in Section~\ref{sec:projection}, and state a useful lemma for studying the norm of linear operators on Schauder spaces in Section~\ref{sec:operator_norm}. We then recall a specific variation of the Newton-Kantorovich theorem, which is a cornerstone of many computer-assisted technique, in Section~\ref{sec:NK}, and then present a very easy example where we use this theorem to validate a gPC expansion in Section~\ref{sec:BasicExample}.

\subsection{Fourier coefficients and $\ell^1$ spaces}
\label{sec:Fourier}

It will be convenient to represent several of the solutions we look for in this work as Fourier series, such as
\begin{align*}
u(t) = \sum_{k\in\Z} u_k e^{ikt}.
\end{align*}
A natural function space to work with is then to consider the set of Fourier coefficients having some prescribed decay rate.
\begin{definition}
\label{def:ell_1_Z}
Let $\left(X,\Vert \cdot\Vert\right)$ be a normed vector space and $\nu\geq 1$. We define
\begin{align*}
\ell^1_\nu(\Z,X) = \left\{ u=(u_n)_{n\in\Z} \in X^\Z,\ \left\Vert u\right\Vert_{\ell^1_\nu(\Z,X)} := \sum_{n\in\Z} \Vert u_n\Vert \nu^{\vert n \vert} < \infty \right\}.
\end{align*}
In the sequel, we sometimes shorten $\ell^1_\nu(\Z,X)$ into $\ell^1_\nu$ when knowledge about the set of indices and $X$ is not relevant or clear from context.
\end{definition}
\begin{remark}
As soon as $\nu>1$, the coefficients of an element of $\ell^1_\nu$ decay at least geometrically, which means the associated function has analytic regularity. This might be seen as a strong requirement, but it should rather be thought of as a precise information: if it happens that the solutions we are dealing with  have analytic regularity, by choosing such weighted spaces for the a posteriori analysis we will be able to \emph{prove} that they do have said regularity. If we had to deal with less smooth solutions, we could use different spaces~\cite{LesMir17}.
\end{remark}
We recall that the discrete convolution makes weighted $\ell^1$ spaces into Banach algebras, as soon as the weights are submultiplicative. This Banach algebra property is going to be very useful for obtaining the validation estimates.
\begin{lemma}
\label{lem:BanachBanach}
Let $\left(X,\Vert \cdot\Vert\right)$ be a Banach algebra, with multiplication denoted by $\ast$, and $\nu\geq 1$. Given $u$ and $v$ in $\ell^1_\nu(\Z,X)$, we can define their convolution product $u \circledast v$ by
\begin{align*}
\left(u\circledast v\right)_k = \sum_{l\in\Z} u_l \ast v_{k-l} \qquad \forall~k\in\Z,
\end{align*}
and we have
\begin{align*}
\left\Vert u\circledast v\right\Vert_{\ell^1_\nu} \leq \left\Vert u\right\Vert_{\ell^1_\nu} \left\Vert v \right\Vert_{\ell^1_\nu},
\end{align*}
i.e., $\ell^1_\nu(\Z,X)$ is a Banach algebra for the multiplication $\circledast$.
\end{lemma}

For a deterministic periodic solution, each coefficient $u_k$ is simply a complex number and we will therefore use the above definition with $X=\C$. However, in the presence of a random parameter $p$ in the system, each $u_k$ will also be random, and therefore expressed using a gPC expansion. In that case, $u_k$ will itself be a sequence of coefficients and $X$ an associated sequence space.

In the next subsection, we recall the necessary ingredients for equipping spaces of gPC coefficients with a Banach algebra structure, so as to be able to use the above Lemma.

\subsection{Linearization formulas and generalized convolution products}
\label{sec:gene_convo_prod}

Most of the material presented in this subsection about the relationships between orthogonal polynomials and Banach algebras can be found (in a different context) in the lecture notes~\cite{Szw05}. We also refer to the appendix of~\cite{BreKue20} for discussions related to implementation issues.

In this work, any quantity $x$ (a Fourier coefficient, the period of a periodic orbit, etc) which depends on a parameter $p$ will be written using gPC expansions, i.e.
\begin{align*}
x(p) = \sum_{n\in\N} x_n \phi_n(p),
\end{align*}
where $\left(\phi_n\right)_{n\in\N}$ is basis of polynomials. 

The basis of gPC is that, when $p$ is a random variable having a density function $\varrho_p$ with finite moments, one should use for the basis $\left(\phi_n\right)_{n\in\N}$ orthogonal polynomials with respect to $\varrho_p$, i.e. such that $\int \phi_m \phi_n \varrho_p = 0$ as soon as $m\neq n$.

\begin{example}
Here are a couple of examples, which are particular cases of Jacobi polynomials, that which we make use of in this work
\begin{itemize}
\item The Legendre polynomials $P_n$, which correspond to $\varrho_p(t)=\frac{1}{2}\1_{(-1,1)}(t)$;
\item The Chebyshev polynomials of the first kind $T_n$, which correspond to $\varrho_p(t)=\frac{1}{\pi\sqrt{1-t^2}}\1_{(-1,1)}(t)$;
\item The Chebyshev polynomials of the second kind $U_n$, which correspond to $\varrho_p(t)=\frac{2}{\pi}\sqrt{1-t^2}\1_{(-1,1)}(t)$;
\item The Gegenbauer or ultraspherical polynomials $C_n^{\mu}$, $\mu>-\frac{1}{2}$, $\mu\neq 0$, which correspond to $\varrho_p(t)=\frac{2^{2\mu-1}\mu B(\mu,\mu)}{\pi}(1-t^2)^{\mu-\frac{1}{2}}\1_{(-1,1)}(t)$.
\end{itemize}
We point out that, for a given $\varrho_p$, each orthogonal polynomial $\phi_n$ is only defined up to a multiplicative  constant, and a normalization condition is required in order to uniquely characterize them. In this work, we choose the condition $\phi_n(1) = 1$ for all $n$. With this normalization, we recover the \emph{traditional} definition of the Legendre and Chebyshev polynomials of the first kind, but the \emph{usual} Chebyshev polynomials of the second kind and Gegenbauer polynomials have to be renormalized. The reason behind this normalization choice is explained in Lemma~\ref{lem:gPC_convo}.

Finally, we will also use the monomial basis $\phi_n(p) = p^n$, in order to compare the performances of gPC expansions with the one of Taylor expansions.

For a more complete description of gPC choices and their relations to the Askey scheme, see~\cite{XiuKar02}.
\end{example}

\begin{definition}
\label{def:ell_1_N}
Let $\eta\geq 1$. We define
\begin{align*}
\ell^1_\eta(\N,\C) = \left\{ x=(x_n)_{n\in\N} \in \C^\N,\ \left\Vert u\right\Vert_{\ell^1_\eta(\N,\C)} := \sum_{n\in\N} \vert x_n\vert \eta^{n} < \infty \right\}.
\end{align*}
In the sequel, we always use $\eta$ as a weight when we consider a space of gPC coefficients, and $\nu$ for the Fourier coefficients, in order to better know at a glance which type of object we are currently dealing with.
\end{definition}
\begin{remark}
As in the previous subsection, these spaces encode regularity properties. Indeed, having $\eta\geq 1$ will be sufficient to ensure that, for any $x$ in $\ell^1_\eta$, the corresponding function
\begin{align*}
p\mapsto \sum_{n\in\N} x_n \Phi_n(p),
\end{align*}
is at least continuous (see Lemma~\ref{lem:gPC_convo}), and even analytic when $\eta>1$~\cite[Theorem 8.2]{Tre13}. Thereby, we will often not distinguish between a \emph{sequence} $x=(x_n)_{n\in\N}$ in $\ell^1_\eta$ and the corresponding \emph{function} $x:p\mapsto\sum_{n\in\N} x_n \Phi_n(p)$, and use the same symbol to denote both.
\end{remark}

Given two functions $x(p)$ and $y(p)$ written as gPC expansions, we now want to define a product on the sequence space $\ell^1_\eta(\N,\C)$ corresponding to the multiplication $x(p) y(p)$.

\begin{definition}
\label{def:lin}
Let $\left(\phi_n\right)_{n\in\N}$ be a family of (univariate) real polynomials such that $\phi_n$ is of degree $n$ for all $n$. The linearization coefficients $\left(\alpha^{m,n}_k\right)_{k,m,n\in\N}$ for this family are the real numbers such that
\begin{equation}
\label{eq:lin}
\phi_m\phi_n = \sum_{k=0}^{n+m} \alpha^{m,n}_k \phi_k,\quad \forall~k,n,m \in\N,
\end{equation}
with $\alpha^{m,n}_k = 0$ for all $k>m+n$.
\end{definition}

\begin{definition}
\label{def:convo_gene}
Given linearization coefficients, we define the \emph{generalized convolution product} $\ast$ (associated to the linearization coefficients, or equivalently to the polynomial basis) of two sequences of complex numbers $x=\left(x_n\right)_{n\in\N}$ and $y=\left(y_n\right)_{n\in\N}$ by
\begin{equation*}
\left(x\ast y\right)_k = \sum_{m=0}^\infty\sum_{n=0}^\infty x_m y_n \alpha^{m,n}_k, \qquad \forall~k\in\N.
\end{equation*}
\end{definition}
The generalized convolution product of coefficients corresponds to the pointwise product of functions, at least formally. The Lemma below gives sufficient conditions for the generalized convolution product to be well defined, and for this identification to be justified.
\begin{lemma}
Let $\eta\geq 1$ and $\left(\alpha^{m,n}_k\right)_{k,m,n\in\N}$ be linearization coefficients such that
\begin{align}
\label{eq:sumto1}
\sum_{k=0}^{m+n} \vert \alpha^{m,n}_k\vert = 1, \qquad \forall~m,n\in\N.
\end{align}
Then, for any $x$ and $y$ in $\ell^1_\eta(\N,\C)$, the generalized convolution product $x\ast y$ (associated to the linearization coefficients) is well defined, belongs to $\ell^1_\eta(\N,\C)$, and 
\begin{align*}
\left\Vert x\ast y\right\Vert_{\ell^1_\eta} \leq \left\Vert x\right\Vert_{\ell^1_\eta} \left\Vert y \right\Vert_{\ell^1_\eta},
\end{align*}
i.e., $\ell^1_\eta(\N,\C)$ is a Banach algebra for $\ast$.
\end{lemma}
\begin{proof}
We simply use the triangle inequality and exchanges sums:
\begin{align*}
\left\Vert x\ast y\right\Vert_{\ell^1_\eta} &= \sum_{k\in\N} \left\vert \left(x\ast y\right)_k \right\vert \eta^k \\
&\leq \sum_{m\in\N}\sum_{n\in\N} \vert x_m\vert \eta^m\, \vert y_n\vert \eta^n \sum_{k\in\N} \vert \alpha^{m,n}_k\vert \eta^{k-n-m},
\end{align*}
which allows us to conclude since $\alpha^{m,n}_k = 0$ for $k> m+n$ and $\eta\geq 1$.
\end{proof}

\begin{lemma}
\label{lem:gPC_convo}
Let $\left(\phi_n\right)_{n\in\N}$ be either:
\begin{itemize}
\item the monomial basis,
\item the Legendre polynomials $P_n$,
\item the Chebyshev polynomials of the first kind $T_n$,
\item the Chebyshev polynomials of the second kind $U_n$ (normalized so that $U_n(1)=1$),
\item the Gegenbauer polynomials $C_n^{\mu}$ (normalized so that $C_n^\mu(1)=1$).
\end{itemize}
Then, the associated linearization coefficients satisfy~\eqref{eq:sumto1}. In particular, the corresponding generalized convolution product provides $\ell^1_\eta(\N,\C)$ with a Banach algebra structure.

Moreover, for any $\eta\geq 1$ and any $x=\left(x_n\right)_{n\in\N}$ in $\ell^1_\eta(\N,\C)$, the associated function $x(p) = \sum_{n\in\N} x_n \phi_n(p)$ satisfies
\begin{align*}
\left\Vert x \right\Vert_{\cC^0} := \sup_{p\in[-1,1]} \left\vert x(p)\right\vert \leq \left\Vert x\right\Vert_{\ell^1_\eta}. 
\end{align*}
\end{lemma}
\begin{proof}
The first part of the Lemma is known more generally, for a large class of Jacobi polynomials (see~\cite{Gas70,Gas70bis}), and is based on the nonnegativity of the linearization coefficients. In our context, we have explicit formula for those coefficients in each case~\cite{OlvLozBoiCla10}, which are indeed nonnegative. It then suffices to evaluate~\eqref{eq:lin} at $1$ to get (thanks to the normalization condition)
\begin{align*}
1 = \sum_{k=0}^{m+n} \alpha^{m,n}_k,
\end{align*}
and therefore~\eqref{eq:sumto1}.

The second part of the Lemma is a direct consequence of the fact that, with our choice of normalization, $\left\vert\phi_n(p) \right\vert \leq 1$ for all $p$ in $[-1,1]$, see~\cite{OlvLozBoiCla10}.
\end{proof}

\subsection{Finite dimensional projections}
\label{sec:projection}

In practice, we approximate elements in $\ell^1_\eta(\N,\C)$ by finite dimensional vectors (or equivalently truncated series). 

\begin{definition}
Given $N$ in $\N$, we define the projector $\Pi_N:\C^\N\to\C^\N$ as follows:
\begin{align*}
\left(\Pi_N x\right)_n = 
\left\{
\begin{aligned}
&x_n  \qquad & n< N,\\
&0 \qquad &n\geq N,
\end{aligned}
\right.
\end{align*} 
and
\begin{align*}
\Pi_N \ell^1_\eta(\N,\C) := \left\{ x \in \ell^1_\eta(\N,\C),\ \Pi_N x = x \right\}.
\end{align*}
\end{definition}
We use similar projectors for Fourier series, i.e. $\ell^1$ spaces of sequences indexed by $\Z$. In order not to use too many different notations, we keep the same letter $\Pi$ to also denote these projectors, but make sure to always use the letter $N$ to refer to gPC indices, and $K$ to refer to Fourier indices.
\begin{definition}
Given $K$ in $\N$ and a vector space $X$, we define the projector $\Pi_K:X^\Z\to X^\Z$ as follows:
\begin{align*}
\left(\Pi_K u\right)_k = 
\left\{
\begin{aligned}
&u_k  \qquad & \vert k\vert < K,\\
&0 \qquad &\vert k\vert \geq K,
\end{aligned}
\right.
\end{align*} 
and
\begin{align*}
\Pi_K \ell^1_\eta(\Z,X) := \left\{ u \in \ell^1_\eta(\Z,X),\ \Pi_K u = u \right\}.
\end{align*}
In the case where $X = \ell^1_\eta(\N,\C)$, given $N$ in $\N$ we denote by $\Pi_{K,N}$ the composition of $\Pi_K$ with $\Pi_N$ (applied component-wise). That is, for $u=\left(u_k\right)_{k\in\Z}$ in $\left(\ell^1_\eta(\N,\C)\right)^\Z$,
\begin{align*}
\Pi_{K,N}u = \left(\ldots,0,0,\Pi_N u_{-K+1},\Pi_N u_{-K+2} ,\ldots,\Pi_N u_{K-2},\Pi_N u_{K-1},0,0 \ldots\right),
\end{align*}
and 
\begin{align*}
\Pi_{K,N} \ell^1_\eta(\Z,\ell^1_\eta(\N,\C)) := \left\{ u \in \ell^1_\eta(\Z,\ell^1_\eta(\N,\C)),\ \Pi_{K,N} u = u \right\}.
\end{align*}
\end{definition}

\subsection{Operator norms}
\label{sec:operator_norm}

Controlling operator norms will be crucial in our work, and we are often going to rely on the following statement.
\begin{lemma}
\label{lem:norm_op_split}
Let $X$ be a Banach space with a Schauder basis $\left(e_j\right)_{j\in\N}$ and a norm $\left\Vert \cdot \right\Vert$ of the form
\begin{align}
\label{eq:norm_op}
\left\Vert \sum_{j\in\N} x_j e_j \right\Vert = \sum_{j\in\N} \vert x_j\vert w_j,
\end{align}
for some prescribed weights $w_j>0$. Then, for any bounded linear operator $B$ on $X$, 
\begin{align*}
\left\Vert B\right\Vert = \sup_{j\in\N} \frac{1}{w_j} \left\Vert B e_j\right\Vert.
\end{align*}
Moreover, for any disjointed subsets $I$ and $J$ of $\N$ such that $I\cup J = \N$, if we denote by $X_I$ and $X_J$ the subspaces of $X$ having $\left(e_j\right)_{j\in I}$ and $\left(e_j\right)_{j\in J}$ as Schauder bases,
\begin{align}
\label{eq:norm_op_split}
\left\Vert B \right\Vert = \max\left(\sup_{\substack{x \in  X_I \\ x\neq 0}} \frac{\left\Vert Bx \right\Vert}{\left\Vert x \right\Vert},\ \sup_{\substack{x \in X_J \\ x\neq 0}} \frac{\left\Vert Bx \right\Vert}{\left\Vert x \right\Vert}\right).
\end{align}
\end{lemma}
\begin{proof}
For any $x=\sum_{j\in\N} x_j e_j$ in $X$, we simply use the triangle inequality to get
\begin{align*}
\left\Vert Bx \right\Vert &\leq \sum_{j\in\N} \vert x_j\vert \left\Vert Be_j\right\Vert \leq \left(\sup_{j\in\N} \frac{1}{w_j} \left\Vert B e_j\right\Vert \right) \sum_{j\in\N} \vert x_j\vert w_j,
\end{align*}
hence $\left\Vert B\right\Vert \leq \sup_{j\in\N} \frac{1}{w_j} \left\Vert B e_j\right\Vert$. However, for any $j$ in $\N$,
\begin{align*}
\frac{1}{w_j} \left\Vert B e_j\right\Vert = \frac{\left\Vert B e_j\right\Vert}{\left\Vert e_j\right\Vert} \leq \left\Vert B\right\Vert,
\end{align*}
which proves~\eqref{eq:norm_op}. The identity~\eqref{eq:norm_op_split} then simply amounts to
\begin{equation*}
\sup_{j\in\N} \frac{1}{w_j} \left\Vert B e_j\right\Vert = \max\left( \sup_{j\in I} \frac{1}{w_j} \left\Vert B e_j\right\Vert,\, \sup_{j\in J} \frac{1}{w_j} \left\Vert B e_j\right\Vert \right). \hfill \qedhere
\end{equation*}
\end{proof}

\subsection{A kind of Newton-Kantorovich theorem}
\label{sec:NK}

As is the case in many works on validated numerics, a crucial tool in our argument is a kind of Newton-Kantorovich theorem~\cite{Ort68}, which allows us to validate a posteriori a numerically obtained solution. 

Given a map $\F$ defined on a Banach space $\X$, and an \emph{approximate} zero $\bx$ of $\F$, this theorem provides us with sufficient conditions guaranteeing the existence of a \emph{genuine} zero $x^*$ of $\F$ near $\bx$, together with explicit error bounds between $\bx$ and $x^*$.

\begin{theorem}
\label{th:NK}
Let $\X$ and $\Y$ be Banach spaces, $\F$ be a $\cC^1$ map from $\X$ to $\Y$, $\bx$ an element of $\X$, $A$ a linear injective map from $\Y$ to $\X$, and $r^*$ in $(0,+\infty]$. Assume there exist nonnegative constants $Y$, $Z_1$ and $Z_2$ such that
\begin{subequations}
\label{eq:YZ12}
\begin{align}
\left\Vert A\F(\bx) \right\Vert_\X &\leq Y \label{eq:Y}\\
\left\Vert I - AD\F(\bx) \right\Vert_\X  &\leq Z_1 \label{eq:Z1}\\
\left\Vert A\left(D\F(x) - D\F(\bx)\right) \right\Vert_\X &\leq Z_2 \left\Vert x-\bx \right\Vert_\X \qquad \forall~x\in\B_\X(\bx,r^*), \label{eq:Z2}
\end{align}
\end{subequations}
where $D\F$ denotes the Fr\'echet derivative of $\F$, $\left\Vert\cdot\right\Vert_\X$ simultaneously denotes the norm on $\X$ and the associated operator norm, and $\B_\X(\bx,r^*)$ is the closed ball of center $\bx$ and radius $r^*$ in $\X$.
If these constants satisfy
\begin{subequations}
\label{eq:cond}
\begin{align}
Z_1 &< 1 \label{eq:cond1}\\
2YZ_2 &< (1-Z_1)^2, \label{eq:cond2}
\end{align}
\end{subequations}
then, for any $r$ satisfying 
\begin{align}
\label{eq:r}
\frac{1-Z_1 - \sqrt{(1-Z_1)^2-2YZ_2}}{Z_2} \leq r < \min\left(\frac{1-Z_1}{Z_2}, r^*\right)
\end{align}
there exists a unique zero $x^*$ of $\F$ in $\B_\X(\bx,r)$.
\end{theorem}
As previously mentioned, similar results already appeared many times, especially in the computer-assisted proof literature (see, e.g.,~\cite{AriKocTer05,DayLesMis07,Plu92,Yam98}), and we refer to~\cite{BerBreLesVee21} for a detailed proof, which merely consists in applying the contraction mapping theorem to $x\mapsto x - D\F(\bx)^{-1} \F(x)$.

\begin{remark}
In practice, applying this theorem requires two main ingredients: a good enough approximate solution $\bx$ so that $Y$ is small enough, and a good enough approximate inverse $A$ of $D\F(\bx)$ so that~\eqref{eq:cond1} holds. For a given $Z_1$ satisfying~\eqref{eq:cond1} and $Z_2$, the condition~\eqref{eq:cond2} tells us in a quantitative way how small $Y$ has to be (and therefore, in some sense, how good of an approximate solution $\bx$ has to be), for the existence of a nearby true solution to be guaranteed. Let us also mention that the injectivity assumption for $A$ is usually automatically satisfied as soon as~\eqref{eq:Z1} and~\eqref{eq:cond1} hold, thanks to some structural properties of $A$, as we will see whenever we apply Theorem~\ref{th:NK} in this work.

Conceptually, defining a suitable $A$, which is not only a good approximate inverse of $D\F(\bx)$ but also \emph{simple} enough that all the estimates~\eqref{eq:YZ12} can be obtained, is often the crucial part. For deterministic problems, say a periodic orbit in~\eqref{eq:RDE_intro} for a given value of $p$, this $A$ is often defined as a finite rank perturbation of a somewhat simple operator (for instance a diagonal operator). In this work, we will see how to generalize this construction to parameter dependent problems.

Finally, let us point out that the codomain $\Y$ of $\F$ is inconsequential, as it does not appear anywhere in the estimates~\eqref{eq:YZ12}. The only thing that really matters is that the composition $A\F$ does map $\X$ into itself.
\end{remark}

\subsection{A basic example}
\label{sec:BasicExample}

In this subsection, we showcase on a very simple example how the Banach algebra structure of gPC expansions presented in Section~\ref{sec:gene_convo_prod} and Theorem~\ref{th:NK} can be combined to provide a fully rigorous uncertainty quantification for an algebraic problem.

Let $p$ be a uniform random variable on $[-1,1]$. Assume we are given a function $g$ as a truncated Legendre series:
\begin{align}
\label{eq:g_ex}
g(p) = \sum_{n=0}^{N_g-1} g_n P_n(p),
\end{align}
with some given $N_g$ and coefficients $\left(g_n\right)_{0\leq n\leq N_g}$ such that $g$ is positive on $[-1,1]$, and that we want to compute $x=x(p) = \sqrt{g(p)}$. For a given $g$, and a given truncated Legendre series $\bx=\bx(p)$ approximating $\sqrt{g(p)}$, we are going to derive a fully computable a error bound between this approximate gPC representation and the true object $\sqrt{g(p)}$. While this example in itself is of limited interest, the techniques we use to solve it generalize very well to more complicated and interesting problems, in particular when we do not have a closed-form expression for $x$ in terms of $p$, as we will see in the remaining sections of the paper.

\begin{proposition}
\label{prop:ex}
Consider the function $g$ of the form~\eqref{eq:g_ex} with $N_g=6$, $g_0=2$, $g_1=-1$, $g_5 = -1/2$ and $g_n=0$ otherwise. Let $\bx=\bx(p) = \sum_{n=0}^{5} \bx_n \phi_n(p)$, where the coefficients $\bx_n$ are given in Table~\ref{tab:ex}. Both $g$ and $\bx$ are represented in Figure~\ref{fig:ex}. 

There exists a unique $x^*$ in $\ell^1_1(\N,\R)$ satisfying $x^*(p) = \sqrt{g(p)}$ for all $p$ in $[-1,1]$ and $\left\Vert \bx - x^*\right\Vert_{\ell^1_1} \leq 0.074$.
\end{proposition}
\begin{proof}
We consider $\eta= 1$, and the map $F:\ell^1_\eta(\N,\R) \to \ell^1_\eta(\N,\R)$ defined by
\begin{align*}
F(x) = x\ast x - g  \qquad \forall~x\in\ell^1_\eta,
\end{align*}
where we identify the function $g$ with its sequence of Legendre coefficients, and $\ast$ is the generalized convolution product associated to the Legendre basis.

We are going to apply Theorem~\ref{th:NK} to the map $F$, with $\X=\Y=\ell^1_\eta(\N,\R)$, and $\bx$ as an approximate solution (in practice $\bx$ was obtained by using Newton's method on the finite dimensional projection $\Pi_{N_g}\, F\, \Pi_{N_g}$ of $F$). In order to do so, we first need to define a linear map $A$ on $\ell^1_\eta(\N,\R)$, which should be a \emph{good enough} approximate inverse of $DF(\bx)$, in the sense that~\eqref{eq:Z1} should be satisfied with $Z_1<1$. Since $DF(\bx)$ is nothing but the multiplication operator $x\mapsto (2\bx)\ast x$, we compute numerically an element $a$ in $\Pi_{N_g}\ell^1_\eta$ such that $2\bx\ast a \approx 1$ (see Table~\ref{tab:ex}), and define $A$ as the multiplication operator by $a$, i.e.
\begin{align*}
A x = a\ast x \qquad \forall~x\in\ell^1_\eta.
\end{align*}
We are now ready to derive bounds $Y$, $Z_1$ and $Z_2$ satisfying assumption~\eqref{eq:YZ12}.
\begin{itemize}
\item For the $Y$ bound, we simply have
\begin{align*}
AF(\bx) = a\ast\left(\bx\ast\bx - g \right).
\end{align*}
Since all the elements involved, namely $a$, $\bx$ and $g$, belong to $\Pi_{N_g} \ell^1_\eta$, or equivalently are polynomials of degree at most $N_g-1$, $AF(\bx)$ is a polynomial of degree at most $3(N_g-1)$. Its coefficients, and therefore its norm, can thus be computed explicitly, and we can take $Y = \left\Vert a\ast\left(\bx\ast\bx - g \right) \right\Vert_{\ell^1_\eta}$.
\item For the $Z_1$ bound, notice that $I-ADF(\bx)$ is simply the multiplication operator by $1-a\ast(2\bx)$. Therefore its operator norm is equal to the norm of $1-a\ast(2\bx)$, which can also be computed explicitly, and we can take $Z_1 = \left\Vert 1-2a\ast\bx \right\Vert_{\ell^1_\eta}$.
\item Finally, for the $Z_2$ bound, $A\left(D\F(x) - D\F(\bx)\right)$ is the multiplication operator by $2a\ast(x - \bx)$, therefore we can take $Z_2 = \left\Vert 2a \right\Vert_{\ell^1_\eta}$ and $r^*=+\infty$.
\end{itemize}
In principle, one could evaluate the obtained bounds $Y$, $Z_1$ and $Z_2$ for $a$, $\bx$ and $g$ given in Table~\ref{tab:ex} by hand, and then check whether the conditions~\eqref{eq:cond} hold. To do it by hand would of course be a waste of time, and become very impractical for higher dimensional problems. Therefore, we evaluate these bounds with a computer, but with interval arithmetic~\cite{Moo79,Tuc11} rather than the usual floating point arithmetic, so as to control rounding errors and ensure that the numbers we obtain for $Y$, $Z_1$ and $Z_2$ do satisfy~\eqref{eq:YZ12}. We obtain
\begin{align*}
Y = 0.06509919865113,\quad Z_1 = 0.07544522585888,\quad Z_2 = 1.18378588092726.
\end{align*}
Hence~\eqref{eq:cond1} holds, and we check that~\eqref{eq:cond2} holds as well. Moreover, the definition of $Z_1$ together with~\eqref{eq:cond1} imply that $a\ast\bx$ is invertible, therefore so is $a$ (because $\ell^1_\eta$ is a commutative ring), and we do have that $A$ in injective. We can thus apply Theorem~\ref{th:NK}, and according to~\eqref{eq:r}, there exists a unique zero $x^*$ of $F$ in $\B_{\ell^1_\eta}(\bx,r)$ for all $r\in[r_{min},r_{max})$, with
\begin{align*}
r_{min} = 0.073908425226395,\qquad r_{max} = 0.781015206412929.
\end{align*}
The computational parts of the proof can be reproduced by downloading the Matlab code at~\cite{BreGit}, and running \texttt{script\_BasicExample.m} (in order to get a rigorous proof including rounding error control, you also need Intlab~\cite{Rum99}).
\end{proof}

\begin{figure}[!h]
\begin{floatrow}
\ffigbox{
	\includegraphics[scale=0.4]{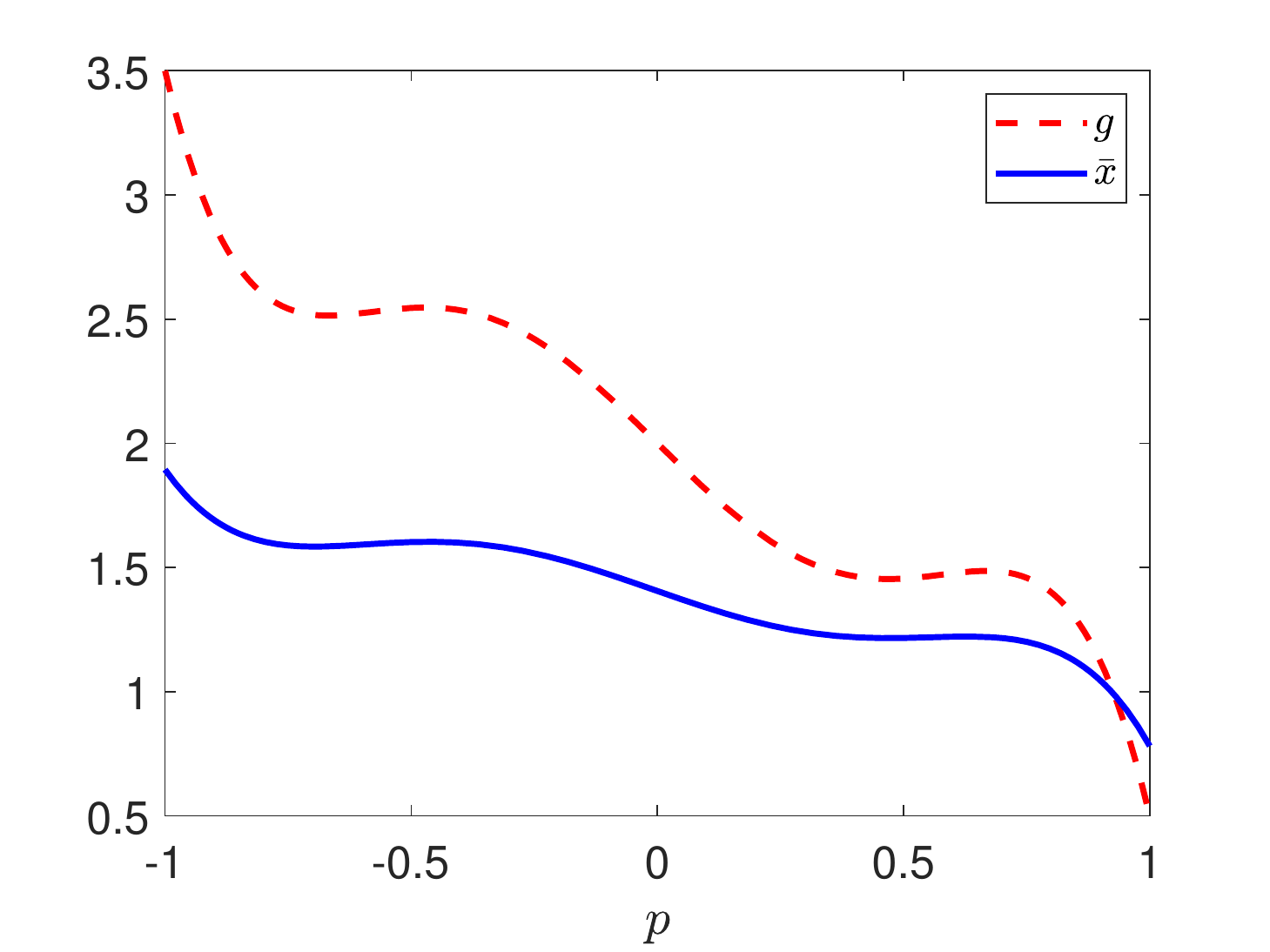}
}
{
	\caption{The input function $g$ and the approximate solution $\bar{x}$ which is validated a posteriori in Proposition~\ref{prop:ex}.}
	\label{fig:ex}
}
\capbtabbox{
	\centering
	\begin{tabular}{c||c|c|}
	$n$ & $\bx$ & $a$ \\
	\midrule[2pt]
	0 & 1.397466142483791 & 0.367312683971841\\
	\midrule
	1 & -0.361966926543100 & 0.100133468883877\\
	\midrule
	2 & -0.034437671606020 & 0.030295726801934\\
	\midrule
	3 & -0.010225992059514 & 0.014887942068709\\
	\midrule
	4 & -0.025424456095587 & 0.023944696446413\\
	\midrule
	5 & -0.184395889766637 & 0.055318422290860\\
	\bottomrule[2pt]
	\end{tabular}
}
{
	\caption{The coefficients of $\bx$ and $a$ used in Proposition~\ref{prop:ex} and in its proof.}
	\label{tab:ex}
}
\end{floatrow}
\end{figure}

\begin{remark}
Let us assume again that $p$ is a uniform random variable on $[-1,1]$. If we had only computed $\bx$ numerically as an approximation of $\sqrt{g(p)}$, we would ``know'' for instance that
\begin{align*}
\E\left(\sqrt{g(p)}\right) \approx \E\left(\bx(p)\right) = \bx_0 = 1.397466...,
\end{align*} 
but without any guarantee regarding the precision of the approximation. With Proposition~\ref{prop:ex}, we get a fully guaranteed (even rounding errors are accounted for) a posteriori estimate controlling the distance between $\sqrt{g(p)}$ and $\bx$, which proves that 
\begin{align*}
\left\vert \E\left(\sqrt{g(p)}\right) - \bx_0 \right\vert \leq r_{min}\approx 0.074.
\end{align*}
It turns out that this estimate is rather conservative, but more precise results can be obtained by looking for $\bx$ (and then for $a$), in a larger dimensional subspace. For instance, if we look for $\bx$ in $\Pi_N\ell^1_\eta$ with $N=20$ instead of $N=6$, we already get an error bound $r_{min}$ of the order of $10^{-4}$, and with $N=50$ we get again a much more precise approximation, together with a validation radius $r_{min}$ of less than $10^{-8}$, which tells us in particular that
\begin{align*}
\E\left(\sqrt{g(p)}\right) = 1.39729844 \pm 10^{-8}.
\end{align*}
For more details, simply run \texttt{script\_BasicExample.m} with different values of $N$.
\end{remark}

\section{A posteriori validation of parameter-dependent periodic orbits for the Lorenz system}
\label{sec:Lorenz}

In this section, we study periodic orbits in the Lorenz system

\begin{equation}
\label{eq:Lorenz}
\left\{\begin{aligned}
x'&=\sigma(y-x) \\
y'&=\rho x-y-xz \\
z'&=-\beta z +xy,
\end{aligned}\right.
\end{equation}
with the \emph{usual} parameter values $\sigma = 10$, $\beta = 8/3$, but assuming $\rho$ is of the form
\begin{align}
\label{eq:Lorenz_rho}
\rho = \brho + \delta p,
\end{align}
where $\brho=28$, $\delta\geq 0$ is a given constant, and $p$ varies in $[-1,1]$. We can either think of $p$ as being a random variable taking values in $[-1,1]$, in which case the $p$-depending periodic orbits are random periodic orbits, or consider a deterministic continuation problem in $p$.

\subsection{Setting}
\label{sec:Setting_Lorenz}

We follow the framework developed in~\cite{BreKue20} for the approximation of random periodic orbits based on Fourier$\times$gPC expansions, and extend it to include the a posteriori validation of the obtained truncated expansions. We look for $x$, $y$ and $z$ of the form
\begin{align}
\label{eq:exp_Lorenz_Four}
x(t,p) = \sum_{k\in\Z} x_k(p) e^{\txti\Omega(p) t},
\end{align}
where each Fourier coefficient $x_k(p)$ as well as the unknown frequency $\Omega(p)$ are also expanded as series, this time with the appropriate gPC basis
\begin{align}
\label{eq:exp_Lorenz_gPC}
x_k(p) = \sum_{n\in\N} x_{k,n} \phi_n(p),\qquad \Omega(p) = \sum_{n\in\N} \Omega_n \phi_n(p).
\end{align}

A natural space in which to apply Theorem~\ref{th:NK} so as to validate an approximate solution $\bX=(\bO,\bx,\by,\bz)$ is then given by
\begin{align}
\label{eq:X_Lorenz}
\X = \ell^1_\eta(\N,\C) \times \left(\ell^1_\nu\left(\Z,\ell^1_\eta(\N,\C)\right)\right)^3,
\end{align}
for some $\nu,\eta\geq 1$ to be specified later. For $X=(\Omega,x,y,z)$ in $\X$, we consider the norm
\begin{align*}
\left\Vert X \right\Vert_\X := \vert \Omega\vert + \left\Vert x \right\Vert_{\ell^1_\nu(\Z,\ell^1_\eta(\N,\C))} + \left\Vert y \right\Vert_{\ell^1_\nu(\Z,\ell^1_\eta(\N,\C))} + \left\Vert z \right\Vert_{\ell^1_\nu(\Z,\ell^1_\eta(\N,\C))}.
\end{align*}
We assume that the linearization cofficients of the chosen gPC basis satisfy~\eqref{eq:sumto1}, so that Lemma~\ref{lem:gPC_convo} and then Lemma~\ref{lem:BanachBanach} apply, i.e. $\ell^1_\eta(\N,\C)$ is a Banach algebra for the generalized convolution product $\ast$, and then $\ell^1_\nu\left(\Z,\ell^1_\eta(\N,\C)\right)$ is itself a Banach algebra for the convolution product $\circledast$.

Introducing the linear operator $\K$, which to any sequence $x\in\ell^1_\nu\left(\Z,\ell^1_\eta(\N,\C)\right)$ associates the sequence $\K x$ defined as 
\begin{align}
\label{eq:K}
\left(\K x\right)_k = k x_k \qquad \forall~k\in\Z,
\end{align} 
and plugging the Ansatz~\eqref{eq:exp_Lorenz_Four}-\eqref{eq:exp_Lorenz_gPC} into the Lorenz system~\eqref{eq:Lorenz}, we can rewrite the resulting set of equations on the Fourier$\times$gPC coefficients in the form
\begin{align*}
F(X) = 0,
\end{align*}
where $X=(\Omega,x,y,z)$ belongs to $\X$, and $F(X) = \left(F^{(x)}(X),F^{(y)}(X),F^{(z)}(X)\right)$ with
\begin{equation}
\label{eq:F_Lorenz}
\left\{\begin{aligned}
F^{(x)}(X) &= -\txti \K \left(\Omega \circledast x\right) - \sigma x +\sigma y\\
F^{(y)}(X) &= -\txti \K \left(\Omega \circledast y\right) + \rho \circledast x  - y - \left(x\circledast z\right)  \\
F^{(z)}(X) &= -\txti \K \left(\Omega \circledast z\right) -\beta z + \left(x\circledast y\right).
\end{aligned}\right.
\end{equation}
In the above equations, we identify an element of $\ell^1_\eta(\N,\C)$ like $\Omega$ or $\rho$ with its natural injection in $\ell^1_\nu\left(\Z,\ell^1_\eta(\N,\C)\right)$, which allows us to write for instance $\Omega \circledast x$, which is nothing but the sequence $\left(\Omega\ast x_k\right)_{k\in\Z}$, where each $\Omega\ast x_k$ is an element of $\ell^1_\eta(\N,\C)$.

This $F$ is almost the one to which we will apply Theorem~\ref{th:NK}, but we first need to add a phase condition to get rid of time-translation invariance and allow $F$ to have isolated zeros. Here we depart slightly from the framework introduced in~\cite{BreKue20}, in which we used a Poincaré phase (or transversality) condition, and instead impose
\begin{align}
\label{eq:G}
G(X) := \sum_{k\in\Z} ik\left(x_k \ast \conj(\tilde x_k)  + y_k \ast \conj(\tilde y_k)  + z_k \ast \conj(\tilde z_k)  \right) = 0,
\end{align}
where $(\tilde x,\tilde y,\tilde z)$ is an approximate solution previously computed. This is inspired from the integral phase condition
\begin{align*}
\int \langle u, \tilde u ' \rangle = 0,
\end{align*}
which has proven to be more robust numerically~\cite{KraOsiGal07}, and turns out to also be more efficient regarding the a posteriori validation.

Given truncation levels $K$ and $N$ in $\N$, and an approximate periodic solution $\bX=(\bO,\bx,\by,\bz)$ in $\Pi_{K,N} \X$, we are going to try and validate this approximation using Theorem~\ref{th:NK} for the map 
\begin{align}
\label{eq:FF_Lorenz}
\F=(G,F)
\end{align}
and the space $\X$. In order to do so, we first need to derive a suitable approximate inverse $A$ of $D\F(\bX)$, and then to obtain bounds satisfying~\eqref{eq:YZ12}. We accomplish these tasks in the next two subsections.

\subsection{The approximate inverse $A$}
\label{sec:A}

If we were considering a deterministic periodic orbit (for a given value of $p$), and therefore working with the space $\X_{deter} = \C \times \left(\ell^1_\nu\left(\Z,\C\right)\right)^3$ rather than $\X = \ell^1_\eta(\N,\C) \times \left(\ell^1_\nu\left(\Z,\ell^1_\eta(\N,\C)\right)\right)^3$, a typical way to construct $A$ would be as follows. One would split the space into a \emph{finite part} and a \emph{tail part}
\begin{align*}
\X_{deter} = \Pi_K \X_{deter} \oplus (I-\Pi_K)\X_{deter},
\end{align*}
where $\Pi_K \X_{deter} = \C \times \left(\Pi_K \ell^1_\nu\left(\Z,\C\right)\right)^3$, and define $A$ separately on both subspaces. For the finite part, we would simply compute numerically an inverse $A_{K}$ of $\Pi_K D\F(\bX)\Pi_K$ which can be represented as a $(6K-2)\times(6K-2)$ matrix with complex entries. For the tail part, i.e. the higher order modes, the parts of $\F$ corresponding to the differential operator would be the most important one, and we would neglect the rest to define $A$:
\begin{align*}
A X_k := -\frac{1}{ik\bO} X_k, \qquad\forall~\vert k\vert \geq K,
\end{align*}
where $X_k = (x_k, y_k, z_k)$. This is but an example of a general strategy for defining approximate inverses in the context of computer-assisted proofs, which consists in choosing $A$ as a finite rank perturbation of some leading order operator that can be inverted by hand, the finite rank part being directly related to a finite dimensional projection of the map $\F$.

In this work, for a random periodic orbit, we adopt this strategy with a slight twist, by trying to mimic as much a possible the situation in the deterministic case, but replacing $\C$ by $\ell^1_\eta(\N,\C)$ and the multiplication on $\C$ by the generalized convolution product $\ast$.

\medskip

We consider the same splitting as above, based only on the Fourier modes
\begin{align*}
\X = \Pi_K \X \oplus (I-\Pi_K)\X,
\end{align*}
where $\Pi_K \X = \ell^1_\eta(\N,\C) \times \left(\Pi_K \ell^1_\nu\left(\Z,\ell^1_\eta(\N,\C)\right)\right)^3$. We will still refer informally to both subspaces as the \emph{finite} part and the \emph{tail} part respectively, but we emphasize that $\Pi_K \X$ is only ``finite'' in terms of Fourier modes, but remains an infinite dimensional subspace because we did not truncate anything in the gPC components. 

\begin{remark}
It is really crucial to take the ``finite'' part, i.e. the part on which the inverse will be computed accurately, as $\Pi_K\X$ and not as $\Pi_{K,N}\X$, otherwise the resulting $A$ will not be a good enough approximate inverse. Indeed, we are allowed to truncate in Fourier because of the regularizing properties of the equation in $t$ (which corresponds to the Fourier expansion), but there is no such regularization in $p$ (which corresponds to the gPC expansion).
\end{remark}

Regarding the tail part, we first compute numerically an approximate inverse $\Upsilon$ of $\bO$ in $\Pi_N\ell^1_\eta(\N,\C)$, i.e. such that $\Upsilon\ast\bO \approx 1$. Then, we define $A$ in the tail in a similar way as above, except $\frac{1}{\bO}X_k$ now becomes $\Upsilon \ast X_k$. 

For the finite part, we will also compute numerically an approximate inverse $A_K$ of $\Pi_K D\F(\bX)\Pi_K$. However, $\Pi_K D\F(\bX)\Pi_K$ is no longer finite dimensional, but can be identified with a linear operator on $\left(\ell^1_\eta(\N,\C)\right)^{6K-2}$. To make things slightly more concrete, this means we can still represent $\Pi_K D\F(\bX)\Pi_K$ as a $(6K-2)\times(6K-2)$ matrix, except each entry is now a linear operator on $\ell^1_\eta(\N,\C)$ rather than a complex number. The key point here is that each of these linear operators is not any linear operator, but a multiplication operator, and can therefore be represented compactly by an element of $\ell^1_\eta(\N,\C)$ (analogously to the way complex numbers in the deterministic case actually represent multiplication operators on $\C$). Therefore, we compute an approximate inverse $A_K$ of $\Pi_K D\F(\bX)\Pi_K$ under the form of a $(6K-2)\times(6K-2)$ matrix of multiplication operators on $\ell^1_\eta(\N,\C)$, represented by $(6K-2)^2$ elements of $\Pi_N\ell^1_\eta(\N,\C)$. Each of these multiplication operator is still of infinite rank, but the fact they are multiplications (generalized convolutions) with elements of $\Pi_N\ell^1_\eta(\N,\C)$ means that $A_K$ can be represented and stored on a computer.

To summarize, we define the linear operator $A$ by
\begin{align}
\label{eq:A_Lorenz}
\left\{
\begin{aligned}
&A \Pi_K(X) = A_K \Pi_K X \\
&A X_k = \frac{1}{-\txti k} \Upsilon \ast X_k,\qquad \vert k\vert \geq K,
\end{aligned}
\right.
\end{align}
where $\Upsilon \ast X_k$ must be understood as $\left(\Upsilon \ast x_k,\Upsilon \ast y_k,\Upsilon \ast z_k\right)$, and $A_K$ is a linear operator on $\left(\ell^1_\eta(\N,\C)\right)^{6K-2}$ which takes the form of $(6K-2)^2$ multiplication operators with elements of $\Pi_N\ell^1_\eta(\N,\C)$ (which are computed numerically so that $A_K \approx \left(\Pi_K D\F(\bX)\Pi_K\right)^{-1}$).

\begin{remark}
\label{rem:Lorenz_memory}
In practice, one of the main limiting factors for computer-assisted proofs like the ones we are using here is the dimension of the finite part of $A$, and the computing power and memory requirement associated to it.
Given the type of expansion were are using, namely bi-infinite series, it is remarkable that the number of complex numbers needed to represent this finite part scales likes $K^2N$ (it is actually equal to $(6K-2)^2N$), rather than like $K^2N^2$. This is possible because we take advantage of the multiplication operator structure.
\end{remark}

\subsection{Bounds for the a posteriori validation}
\label{sec:bounds_Lorenz}

Now that $A$ has been defined in~\eqref{eq:A_Lorenz}, we are left with deriving estimates $Y$, $Z_1$ and $Z_2$ satisfying~\eqref{eq:YZ12}. Once a proper framework has been obtained, including an appropriate definition of $A$ and a suitable choice of sequence space, the derivation of these estimates is by now standard in the computer-assisted proof literature. Therefore, we only go into the details when they are specific to the new structure of $A$ that is used in this work.

We recall that the map $\F$ we are considering is defined in~\eqref{eq:F_Lorenz}-\eqref{eq:FF_Lorenz}, the space $\X$ in~\eqref{eq:X_Lorenz}, and that the approximate solution $\bX$ belongs to $\Pi_{K,N}\X$, i.e. is a trigonometric polynomial of degree at most $N-1$, whose coefficients are polynomials of degree at most $K-1$. Similarly, we assume that $\tilde x$, $\tilde y$ and $\tilde z$ involved in the phase condition~\eqref{eq:G} all belong to $\Pi_{K,N}\ell^1_\nu\left(\Z,\ell^1_\eta(\N,\C)\right)$.

\subsubsection{The bound $Y$}

As was the case in the example of Section~\ref{sec:BasicExample}, obtaining a $Y$ bound satisfying~\eqref{eq:Y} is rather straightforward, as we can simply take
\begin{align*}
Y := \left\Vert A\F(\bX) \right\Vert_\X.
\end{align*}
The only thing to notice is that $A\F(\bX)$ has only finitely non-zero coefficients, hence it can be computed exactly on a computer, up to rounding errors which are taken care of by the use of interval arithmetic. Indeed, having $\bX$ in $\Pi_{K,N}\X$ means $\F(\bX)$ belongs to $\Pi_{2K-1,2N-1}\X$, and that $A\F(\bX)$ belongs to $\Pi_{2K-1,3N-2}\X$, hence the above defined $Y$ is computable in finitely many operations.

\subsubsection{The bound $Z_1$}
\label{sec:Z1_Lorenz}

In order to obtain a $Z_1$ estimate, we need to bound the operator norm of $B:=I-AD\F(\bX)$. To that end, we use the following splitting (see Lemma~\ref{lem:norm_op_split})
\begin{align}
\label{eq:normB_Lorenz}
\left\Vert B \right\Vert_\X = \max\left(\sup_{\substack{X \in \Pi_{2K-1} \X \\ X\neq 0}} \frac{\left\Vert BX \right\Vert_\X}{\left\Vert X \right\Vert_\X},\ \sup_{\substack{X \in \left(I-\Pi_{2K-1}\right) \X \\ X\neq 0}} \frac{\left\Vert BX \right\Vert_\X}{\left\Vert X \right\Vert_\X}\right).
\end{align}
In order to handle the first part, we will use the following lemma, which is a direct consequence of the Banach algebra property of Lemma~\ref{lem:gPC_convo} and of the usual computation of $\ell^1$ operator norms.
\begin{lemma}
\label{lem:norm_op}
Let $B$ be a linear operator on $\ell^1_\nu\left(\Z,\ell^1_\eta(\N,\C)\right)$, represented as an \emph{infinite matrix} $\left(B_{k,l}\right)_{k,l\in\Z}$ of linear operators on $\ell^1_\eta(\N,\C)$, and assume that each of those is in fact a multiplication operator by an element $b_{k,l}$ in $\ell^1_\eta(\N,\C)$. Then, denoting by $\left\Vert b \right\Vert_{\ell^1_\eta}$ the infinite matrix of real numbers $\left(\left\Vert b _{k,l}\right\Vert_{\ell^1_\eta} \right)_{k,l\in\Z}$, we have that
\begin{align*}
\left\Vert B \right\Vert_{\ell^1_\nu(\Z,\ell^1_\eta)} &= \left\Vert \left\Vert b \right\Vert_{\ell^1_\eta} \right\Vert_{\ell^1_\nu} = \sup_{l\in\Z} \frac{1}{\nu^{\vert l\vert}} \sum_{k\in\Z} \left\Vert b_{k,l}\right\Vert_{\ell^1_\eta} \nu^{\vert k\vert}.
\end{align*}
\end{lemma}
This lemma easily generalizes to a linear operator defined on $\X$ (or on a subspace of $\X$), and allows us to get a computable upper-bound $Z_1^{finite}$ of the first supremum in~\eqref{eq:normB_Lorenz}. Indeed, for $X$ in $\Pi_{2K-1} \X$ and $B=I-AD\F(\bX)$, $BX$ belongs to $\Pi_{3K-2} \X$, which means we only have finitely many multiplications operators on $\ell^1_\eta$ whose norm we need to compute in order to control this supremum over $\Pi_{2K-1} \X$. Finally, since we assumed that all the multiplication operators in $A$ were with elements of $\Pi_N\ell^1_\eta(\N,\C)$, each of those multiplication operators in $B$ are with elements of $\Pi_{2N-1}\ell^1_\eta(\N,\C)$, which have only finitely many non-zero coefficients, which makes their norm fully computable.

Regarding the second part of the splitting, if $X $ belongs to $\left(I-\Pi_{2K-1}\right) \X$, then $D\F(\bX)X$ is in $\left(I-\Pi_{K}\right) \X$, and therefore so is $AD\F(\bX)X$. Hence, still assuming $X$ belongs to $\left(I-\Pi_{2K-1}\right) \X$, $BX$ is in $\left(I-\Pi_{K}\right) \X$, we can write $BX$ rather explicitly since it does not involve $A_K$: for $\vert k\vert \geq K$ we get
\begin{align*}
\left(BX\right)_k &= X_k - A DF_k(\bX)X \\
&= \begin{pmatrix}
x_k \\ y_k \\ z_k
\end{pmatrix}
+ \frac{1}{\txti k}\Upsilon \ast 
\begin{pmatrix}
-\txti k \bO \ast x_k - \sigma x_k +\sigma y_k\\
-\txti k \bO \ast y_k + \rho \ast x_k  - y_k - \left(\bx\circledast z + x\circledast\bz\right)_k  \\
-\txti k \bO \ast z_k - \beta z_k + \left(\bx\circledast y + x\circledast\by\right)_k
\end{pmatrix} \\
&= \left(1-\Upsilon\ast\bO\right) \ast \begin{pmatrix}
x_k \\ y_k \\ z_k
\end{pmatrix}
+ \frac{1}{\txti k} \Upsilon \ast
\begin{pmatrix}
 -\sigma x_k +\sigma y_k\\
 \rho \ast x_k  - y_k - \left(\bx\circledast z + x\circledast\bz\right)_k  \\
-\beta z_k + \left(\bx\circledast y + x\circledast\by\right)_k
\end{pmatrix},
\end{align*}
which yields
\begin{align*}
\left\Vert BX \right\Vert_\X &\leq \left\Vert 1-\Upsilon\ast\bO\right\Vert_{\ell^1_\eta} \left\Vert X\right\Vert_\X \\
&\quad + \frac{1}{K}\bigg(\left(\sigma\left\Vert\Upsilon\right\Vert_{\ell^1_\eta} + \left\Vert \left(\bz-\rho\right)\circledast\Upsilon \right\Vert_{\ell^1_\nu(\ell^1_\eta)} + \left\Vert \by\circledast\Upsilon\right\Vert_{\ell^1_\nu(\ell^1_\eta)} \right) \left\Vert x \right\Vert_{\ell^1_\nu(\ell^1_\eta)}  \\
&\qquad\qquad  + \left(\left(\sigma + 1\right)\left\Vert\Upsilon\right\Vert_{\ell^1_\eta} + \left\Vert \bx\circledast\Upsilon \right\Vert_{\ell^1_\nu(\ell^1_\eta)} \right) \left\Vert y \right\Vert_{\ell^1_\nu(\ell^1_\eta)} + \left(\left\Vert \bx\circledast\Upsilon \right\Vert_{\ell^1_\nu(\ell^1_\eta)}  + \beta\left\Vert\Upsilon\right\Vert_{\ell^1_\eta} \right) \left\Vert z \right\Vert_{\ell^1_\nu(\ell^1_\eta)} \bigg) \\
&\leq \left(\left\Vert 1-\Upsilon\ast\bO\right\Vert_{\ell^1_\eta} + \frac{1}{K}  \max 
\begin{pmatrix} \sigma\left\Vert\Upsilon\right\Vert_{\ell^1_\eta} + \left\Vert \left(\bz-\rho\right)\circledast\Upsilon \right\Vert_{\ell^1_\nu(\ell^1_\eta)} + \left\Vert \by\circledast\Upsilon\right\Vert_{\ell^1_\nu(\ell^1_\eta)} \\
\left(\sigma + 1\right)\left\Vert\Upsilon\right\Vert_{\ell^1_\eta} + \left\Vert \bx\circledast\Upsilon \right\Vert_{\ell^1_\nu(\ell^1_\eta)} \\
\left\Vert \bx\circledast\Upsilon \right\Vert_{\ell^1_\nu(\ell^1_\eta)} + \beta\left\Vert\Upsilon\right\Vert_{\ell^1_\eta} 
\end{pmatrix}\right)\left\Vert X\right\Vert_\X .
\end{align*}

Therefore, we can define
\begin{align*}
Z_1^{tail} &= \left\Vert 1-\Upsilon\ast\bO\right\Vert_{\ell^1_\eta} + \frac{1}{K}  \max 
\begin{pmatrix} \sigma\left\Vert\Upsilon\right\Vert_{\ell^1_\eta} + \left\Vert \left(\bz-\rho\right)\circledast\Upsilon \right\Vert_{\ell^1_\nu(\ell^1_\eta)} + \left\Vert \by\circledast\Upsilon\right\Vert_{\ell^1_\nu(\ell^1_\eta)} \\
\left(\sigma + 1\right)\left\Vert\Upsilon\right\Vert_{\ell^1_\eta} + \left\Vert \bx\circledast\Upsilon \right\Vert_{\ell^1_\nu(\ell^1_\eta)} \\
\left\Vert \bx\circledast\Upsilon \right\Vert_{\ell^1_\nu(\ell^1_\eta)} + \beta\left\Vert\Upsilon\right\Vert_{\ell^1_\eta} 
\end{pmatrix},
\end{align*}
and finally $Z_1:=\max\left(Z_1^{finite},Z_1^{tail}\right)$ which satisfies~\eqref{eq:Z1}.

\subsubsection{$Z_2$}

As in the example of Section~\ref{sec:BasicExample}, our map $\F$ is quadratic, which allows us to take $r^*=+\infty$. In particular, for any $X$ and $X'$ in $\X$ we have
\begin{align*}
D^2G(\bX)(X,X') & = 0 \\
D^2F(\bX)(X,X') & = i\K \left(\Omega \circledast \begin{pmatrix} x' \\ y' \\ z' \end{pmatrix} + \Omega' \circledast \begin{pmatrix} x \\ y \\ z \end{pmatrix}\right) + \begin{pmatrix}
0 \\ -x\circledast z' - x'\circledast z \\ x\circledast y' + x'\ast y
\end{pmatrix},
\end{align*}
where $\K$ was introduced in Section~\ref{sec:Setting_Lorenz}.
Therefore
\begin{align*}
\left\Vert AD^2\F(\bX)(X,X') \right\Vert_\X &\leq \left\Vert A\K \right\Vert_\X \left(\vert\Omega\vert \left\Vert X' \right\Vert_\X + \vert\Omega'\vert \left\Vert X \right\Vert_\X \right) \\
&\quad  + \left\Vert A \right\Vert_\X \left( \left\Vert x \right\Vert_{\ell^1_\nu(\ell^1_\eta)}  \left(\left\Vert y' \right\Vert_{\ell^1_\nu(\ell^1_\eta)} + \left\Vert z' \right\Vert_{\ell^1_\nu(\ell^1_\eta)}\right) + \left\Vert x' \right\Vert_{\ell^1_\nu(\ell^1_\eta)}  \left(\left\Vert y \right\Vert_{\ell^1_\nu(\ell^1_\eta)} + \left\Vert z \right\Vert_{\ell^1_\nu(\ell^1_\eta)}\right) \right) \\
&\leq \max\left(\left\Vert A\K \right\Vert_\X,\, \left\Vert A \right\Vert_\X\right) \left\Vert X \right\Vert_\X \left\Vert X' \right\Vert_\X,
\end{align*}
and $Z_2 := \max\left(\left\Vert A\K \right\Vert_\X,\, \left\Vert A \right\Vert_\X\right)$ satisfies~\eqref{eq:Z2} (with $r^*=+\infty$).

\begin{remark}
To be precise, in the above definition of $Z_2$ we replace the norm of $A$ and $A\K$ by easily computable upper bounds of their norms, obtained using Lemma~\ref{lem:norm_op}. 

This estimate could also be made slightly sharper, by noticing that some ``columns'' of $A$ and $\K A$ are always multiplied by zero in the above computation of $AD^2\F(\bX)(X,X')$, and can therefore be excluded from the norm computation.
\end{remark}

\subsection{Results}

We are now ready to rigorously validate approximate periodic solutions of~\eqref{eq:Lorenz}-\eqref{eq:Lorenz_rho} represented as truncated Fourier$\times$gPC series, by proving the existence of a true solution within a distance at most $r$ of the approximate one, for an explicit value of $r$.

Approximate solutions using truncated Fourier$\times$gPC series were already obtained in~\cite[Section 6]{BreKue20}, but without guarantee regarding their accuracy, which is what we add in this paper, thanks to Theorem~\ref{th:NK} and the estimates derived up to now in this section. Here is an example of the kind of results we can obtain with this approach.

\begin{theorem}
\label{th:Lorenz}
Take the parameter values $\brho=28$, $\delta=10$ in~\eqref{eq:Lorenz_rho}, the generalized convolution product (Definition~\ref{def:convo_gene}) associated to the Legendre polynomials $P_n$, the weights $\nu=\eta=1$ in the definition of $\X$~\eqref{eq:X_Lorenz}, and the truncation parameters $K=100$ and $N=15$. Consider the approximate Fourier$\times$Legendre solution $\bX$ in $\Pi_{K,N}\X$ of~\eqref{eq:Lorenz}-\eqref{eq:Lorenz_rho}, which can be downloaded at~\cite{BreGit}, and for which a couple of orbits are represented on Figure~\ref{fig:Lorenz}.

There exists a periodic solution $X^*$ in $\X$ of~\eqref{eq:Lorenz}-\eqref{eq:Lorenz_rho}, such that $\left\Vert X^*-\bX\right\Vert_\X \leq r_{min} = 1.3724\times 10^{-4}$. This is the unique solution in the open ball of center $\bX$ and radius $r_{max} = 1.382\times 10^{-3}$ in $\X$.
\end{theorem}
\begin{proof}
We consider $\F$ as in~\eqref{eq:F_Lorenz}-\eqref{eq:FF_Lorenz}, $A$ as in~\eqref{eq:A_Lorenz}, and evaluate the bounds $Y$, $Z_1$ and $Z_2$ (with $r^*=+\infty$) obtained in Section~\ref{sec:bounds_Lorenz}. We get
\begin{align*}
Y = 3.62368...\times 10^{-5} \quad Z_1 = 0.72214... \quad Z_2 = 201.067...,
\end{align*}
hence assumptions~\eqref{eq:cond} are satisfied. From~\eqref{eq:Z1} and \eqref{eq:cond1} we know that $A$ must be surjective, and that the tail part of $A$ is bijective ($Z_1^{tail}<1$ yields $\left\Vert 1-\Upsilon\ast\bO\right\Vert_{\ell^1_\eta}<1$, therefore $\Upsilon$ is invertible). The finite part $A_K$ of $A$ is only surjective a priori, but since $A_K$ can be represented as a (finite) matrix over the commutative ring $\ell^1_\eta$, surjectivity implies injectivity and we do have that $A$ is injective. Theorem~\ref{th:NK} then yields the announced results, with 
\begin{align*}
r_{min} = \frac{1-Z_1 - \sqrt{(1-Z_1)^2-2YZ_2}}{Z_2}\quad\text{and}\quad  r_{max} = \frac{1-Z_1}{Z_2}.
\end{align*}
The computational parts of the proof, namely the computation of the finite part $A_K$ of $A$ and the evaluation of the bounds, can be reproduced using \texttt{script\_Lorenz.m} available at~\cite{BreGit} (with Intlab~\cite{Rum99} for the required interval arithmetic computations). 
\end{proof}
\begin{figure}
\includegraphics[width=0.6\linewidth]{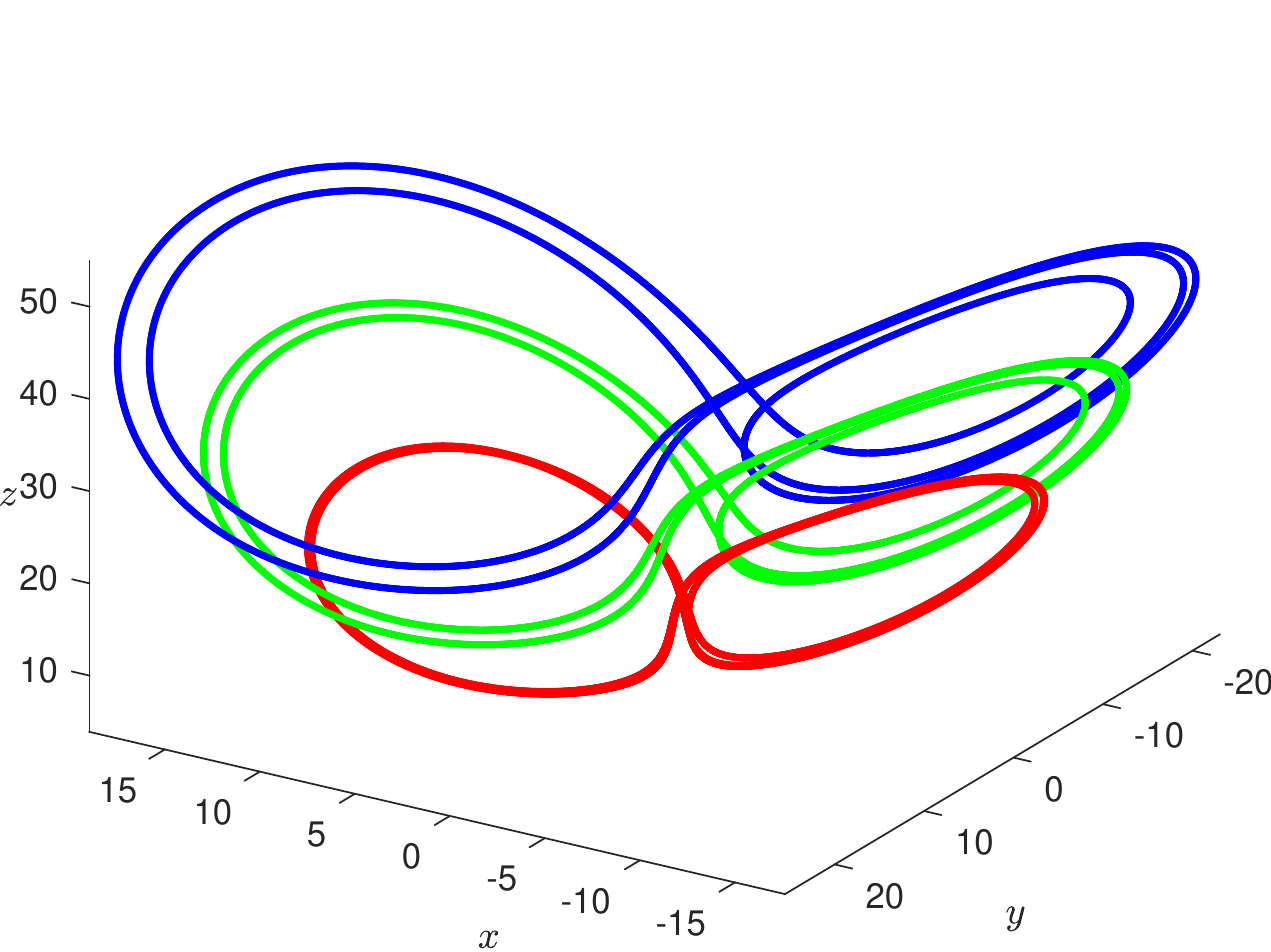}
\caption{Several approximate periodic orbits of the Lorenz system~\eqref{eq:Lorenz}-\eqref{eq:Lorenz_rho}, for $\rho = 18$ ($p=-1$) in red, $\rho = 28$ ($p=0$) in green, and $\rho = 38$ ($p=1$) in blue, all encoded in the Fourier$\times$gPC expansion $\bX$, and validated in Theorem~\ref{th:Lorenz}.}
\label{fig:Lorenz}
\end{figure}
\begin{remark}
Since we used a Legendre expansion in $p$, the solution $X^*$ described in Theorem~\ref{th:Lorenz} would be a natural gPC representation of a random periodic orbit of~\eqref{eq:Lorenz}-\eqref{eq:Lorenz_rho} where $p$ is a uniform random variable in $[-1,1]$. We would then directly get statistics about the random periodic orbit, for instance an approximation of the expectation of its frequency
\begin{align*}
\E \left(\Omega^*\right) \approx \bO_0 = 1.5993...,
\end{align*}
together with a guaranteed error bound
\begin{align*}
\left\vert \E \left(\Omega^*\right) - \bO_0 \right\vert \leq r_{min}.
\end{align*}

Moreover, the obtained solution contains a precise description of a periodic orbit for each $p$ in $[-1,1]$, therefore it could also be used to compute statistics of a random periodic orbit assuming a different distribution for $p$. For instance, if $p$ has a density $\varrho_p$, then we get an approximation of the expectation of its frequency
\begin{align*}
\E \left(\Omega^*\right) \approx \E \left(\bO\right) = \sum_{n=0}^N \bO_n \int_{-1}^1 P_n(s) \varrho_p(s) ds,
\end{align*} 
and an error bound
\begin{align*}
\left\vert \E \left(\Omega^*\right) - \E \left(\bO\right) \right\vert &= \left\vert \sum_{n=0}^\infty \Omega^*_n \int_{-1}^1 P_n(s) \varrho_p(s) ds - \sum_{n=0}^N \bO_n \int_{-1}^1 P_n(s) \varrho_p(s) ds\right\vert \\
&\leq \sum_{n=0}^\infty \left\vert \Omega^*_n - \bO_n\right\vert \int_{-1}^1 \varrho_p(s) ds \\
& = \left\Vert \Omega^* - \bO\right\Vert_{\ell^1_\eta} \\
&\leq r_{min},
\end{align*}
since $\bO_n = 0$ for $n\geq N$ and each $\vert P_n\vert$ is bounded by $1$ on $[-1,1]$.

If $p$ does not have a uniform distribution, the approximate solution obtained using Legendre polynomials will be less accurate (at least in $L^2$ norm) than the one obtained with the gPC basis associated to $\varrho_p$, and one then has to do extra computations a posteriori, like the integrals $\int_{-1}^1 P_n(s) \varrho_p(s) ds$. Nonetheless, since the cost of the validation can change significantly from one choice of basis to the other (see the discussion below), the natural choice of gPC basis (i.e. the one associated to $\varrho_p$) might not always be the cheapest option.
\end{remark}

In the above discussion, we mostly adopted the viewpoint of random periodic orbits, but Theorem~\ref{th:Lorenz} also provides us with a deterministic continuation result, namely the existence (and precise description) of a branch of periodic orbits for $\rho$ going from $\brho-\delta = 18$ to $\brho+\delta = 38$. If there is no underlying random distribution for $p$, we are completely free from the gPC paradigm, and should try to chose the \emph{best} expansion basis, where of course one has to specify in which sense we mean best. In the following we investigate two criteria: 
\begin{itemize}
\item For each basis, what is the smallest value of $N$ for which the validation is successful?
\item For a fixed $N$, what is the minimal validation radius $r_{min}$ obtained which each basis? 
\end{itemize}
The first criterion is related to the cost of the validation, both in terms of computational time and memory requirement (see Remark~\ref{rem:Lorenz_memory}). The second one assesses the accuracy of the obtained approximation, or at least the accuracy that can be guaranteed.

The output of these comparisons is described in Table~\ref{tab:comp_Lorenz_N_var} regarding the cost of the validation, and in Table~\ref{tab:comp_Lorenz_err} regarding the accuracy. In both cases we considered~\eqref{eq:Lorenz}-\eqref{eq:Lorenz_rho} with $\brho=28$ and $\delta=10$, a fixed truncation level $K=100$ for the Fourier modes, and weights $\nu=\eta=1$ in the norm on $\X$. These experiments can be reproduced using \texttt{script\_Lorenz.m} available at~\cite{BreGit} (with Intlab~\cite{Rum99} for the required interval arithmetic computations).

\begin{table}[!h]
\centering
\begin{tabular}{c||c|c|c|c|c}
Polynomial basis & Legendre & Chebyshev  & Chebyshev 2nd kind & Gegenbauer $\mu = 20$ & Taylor \\ 
\midrule[2pt]
Minimal value of $N$ & $14$ & $13$ & $15$ & $22$ & $28$ 
\end{tabular}
\caption{We validate an approximate solution of~\eqref{eq:Lorenz}-\eqref{eq:Lorenz_rho} with $\brho=28$ and $\delta=10$, for several choices of polynomial bases for the expansion in $p$. As soon as $N$ is taken strictly smaller than the value indicated here, the validation fails, typically because~\eqref{eq:cond2} is no longer satisfied.}
\label{tab:comp_Lorenz_N_var}
\end{table}

\begin{table}[!h]
\centering
\begin{tabular}{c||c|c|c|c|c}
Polynomial basis & Legendre & Chebyshev  & Chebyshev 2nd kind & Gegenbauer $\mu = 20$ & Taylor \\ 
\midrule[2pt]
Error bound & $8.7\times 10^{-7}$ & $7.6\times 10^{-7}$ & $9.7\times 10^{-7}$ & $1.5\times 10^{-5}$  & $7.1\times 10^{-4}$ \\
\end{tabular}
\caption{We validate an approximate solution of~\eqref{eq:Lorenz}-\eqref{eq:Lorenz_rho} with $\brho=28$ and $\delta=10$, for several choices of polynomial bases for the expansion in $p$, but this time with a fixed truncation level $N=28$. In each case the validation is successful (meaning that the assumptions of Theorem~\ref{th:NK} can be checked using the estimates of Section~\ref{sec:bounds_Lorenz}), and the displayed error bound corresponds to $r_{min} = \frac{1-Z_1 - \sqrt{(1-Z_1)^2-2YZ_2}}{Z_2}$.}
\label{tab:comp_Lorenz_err}
\end{table}

The Chebyshev polynomials (of the first kind) prove to be the best choice in both metrics (cost and accuracy), which is not surprising given their remarkable approximation properties~\cite{Tre13}, although the difference with the Legendre polynomials or the Chebyshev polynomials of the second kind is barely noticeable. On the other hand, there seems to be a significant difference between using a Chebyshev expansion, and a Gegenbauer expansion (with $\mu=20$) or a Taylor expansion, in particular since the latter require significantly more modes for the validation to be successful (Table~\ref{tab:comp_Lorenz_N_var}), which is related to the fact that they yield less accurate approximations (Table~\ref{tab:comp_Lorenz_err}), at least in the $\X$ norm which is used for the proof. While Taylor expansions were already used successfully to obtain impressive results about validated branches of stationary and perdiodic solutions of PDEs~\cite{Ari22,AriGazKoc21}, the present comparison suggests that replacing the Taylor expansion by a Chebyshev expansion in the continuation variable would prove even more efficient.

\section{Validated continuation of steady states of the Swift-Hohenberg equation}
\label{sec:SH}

We now concentrate fully on the parameter continuation viewpoint, and consider as an example the Swift-Hohenberg equation~\cite{SwiHoh77}
\begin{align}
\label{eq:SH}
\partial_t u = -(1+\Delta)^2 u + \rho u - \beta u^3,
\end{align}
where $u=u(t,x)$ is a scalar function, for which we compute and validate branches of equilibria.
We focus on the 1 dimensional case, where the spatial variable $x$ belongs in $(0,L)$ together with homogeneous Neumann boundary conditions, which is already very rich, as is illustrated by the bifurcation diagram of steady states represented in Figure~\ref{fig:bif_SH}. For the moment we keep $\beta$ fixed, and take $\rho$ as the continuation parameter, which we again normalize by writing
\begin{align}
\label{eq:norm_rho}
\rho = \brho + \delta p,
\end{align}
where $\brho$ and $\delta\geq 0$ are given constants, and $p$ varies in $[-1,1]$.

\begin{figure}[h!]
\includegraphics[width=0.495\linewidth]{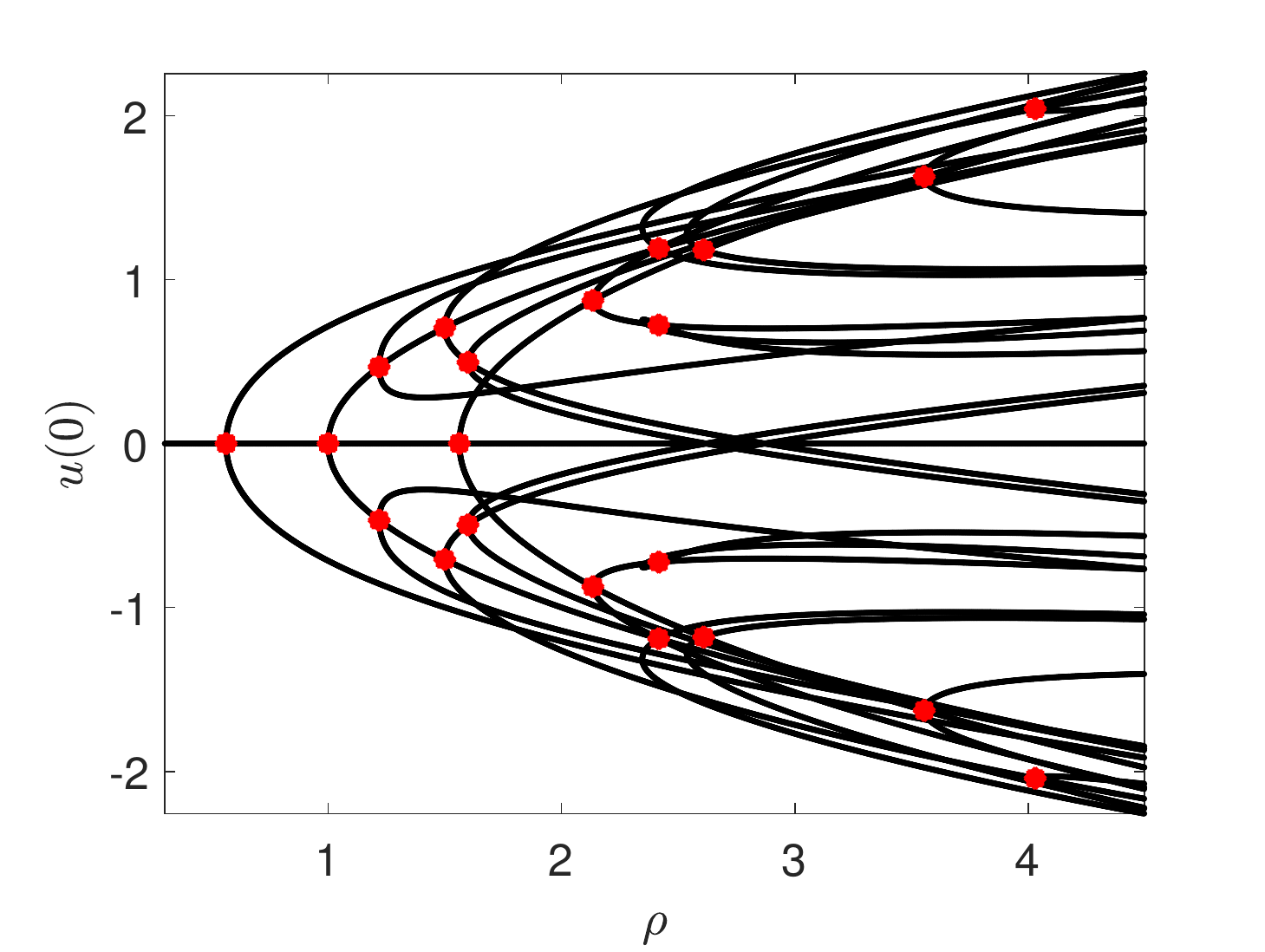} \hfill
\includegraphics[width=0.495\linewidth]{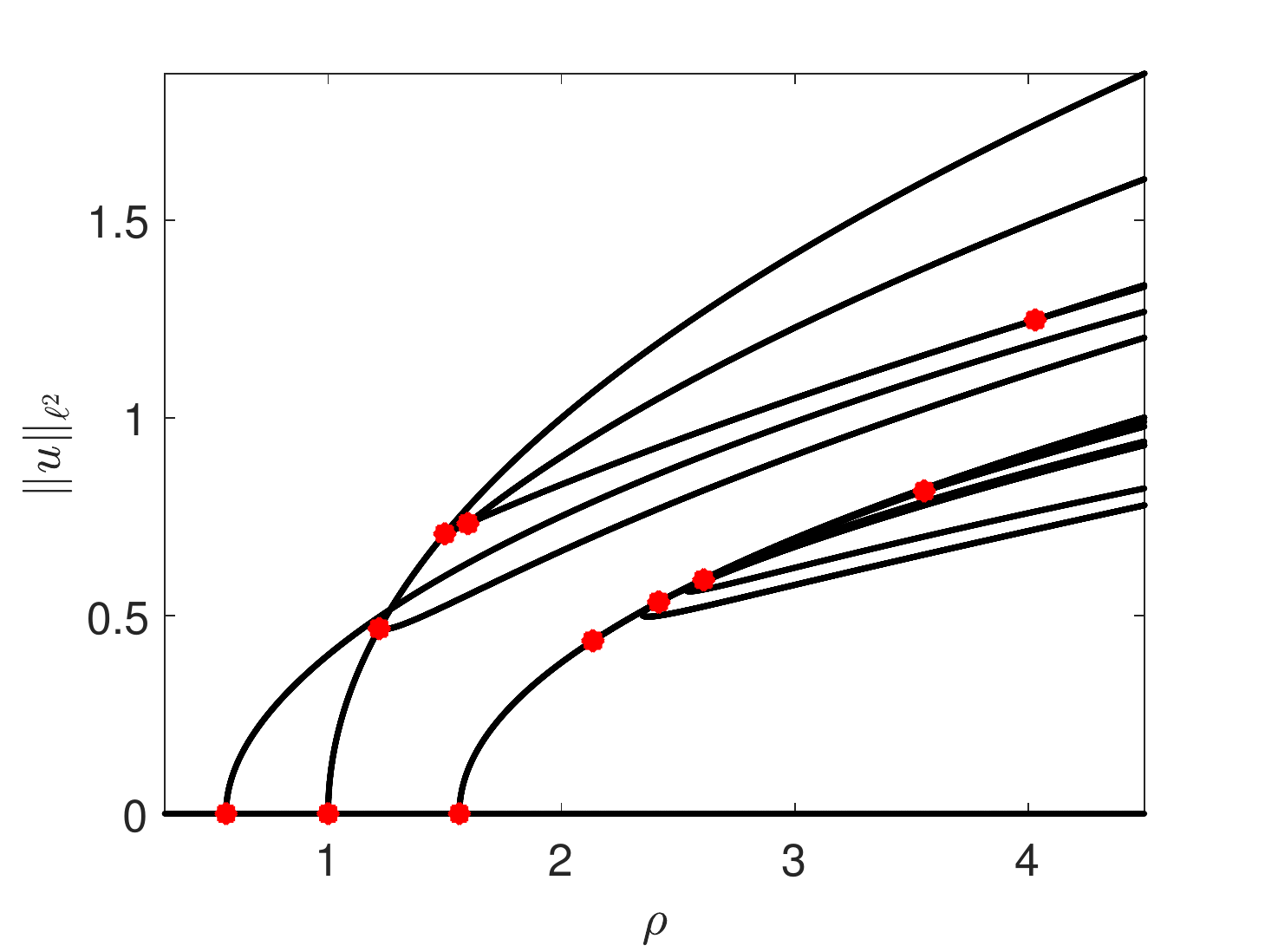}
\caption{A numerical bifurcation diagram of steady states of~\eqref{eq:SH}, for $\beta=1$ and $L=2\pi$, represented using two different projections: $u(0)$ on the left, and $\left\Vert u\right\Vert_{\ell^2}$ on the right. The red dots indicate numerically detected bifurcation.}
\label{fig:bif_SH}
\end{figure}

This problem has already been used as a test case for rigorous continuation methods. We re-emphasize that the main novelty of our work in that regard is the fact that we also expand the solution in the continuation parameter, with several choices of bases, including Chebyshev polynomials, which allows us to represent and validate ``in one go" large portions of the curves of solutions.

\subsection{Setup for the validation}
\label{sec:SH_setup}

The validation setup is very similar to the one used in Section~\ref{sec:Lorenz}, so we go over it more briefly. 

The homogeneous Neumann boundary conditions make it natural to expand the solution in Fourier series in the $x$ variable:
\begin{align}
\label{eq:SH_Four}
u(x,p) = u_0(p) + 2\sum_{k=1}^\infty u_k(p)\cos\left(\frac{k\pi}{L}x\right) = \sum_{k\in\Z} u_k(p) e^{i\frac{k\pi}{L}x},
\end{align}
with $u_{-k} = u_k$, and we use a gPC expansion for the $p$ variable
\begin{align}
\label{eq:up}
u_k(p) = \sum_{n\in\N} u_{k,n} \phi_n(p).
\end{align}

We look for solutions $u=\left(u_{k,n}\right)_{\substack{k\in\Z \\ n\in\N}}$ in the space $\X = \ell^1_\nu\left(\Z,\ell^1_\eta(\N,\R)\right)$ where we impose that $u_{-k,n} = u_{k,n}$ for all $k$ and $n$, and consider the norm
\begin{align*}
\left\Vert u\right\Vert_\X = \left\Vert u\right\Vert_{\ell^1_\nu(\Z,\ell^1_\eta(\N,\R))}.
\end{align*}
The zero-finding map $\F$ is defined as
\begin{align*}
\F(u) = -\left(I-\left(\frac{\pi}{L}\K\right)^2\right)^2 u +\rho\circledast u - \beta u\circledast u \circledast u,
\end{align*}
where the operator $\K$ is defined in equation~\eqref{eq:K}. Given an approximate zero $\bu\in\Pi_{K,N}\X$, we construct an approximate inverse $A$ of $\D\F(\bu)$ as in Section~\ref{sec:A}, the main difference being that the $\frac{-1}{ik}\Upsilon$ factor in the tail part of $A$ is now taken as $\frac{-1}{\lambda_k}$, where
\begin{align*}
\lambda_k := \left(1-\frac{\pi k}{L}\right)^2,
\end{align*}
and we make sure to take the truncation level $K$ large enough to ensure that $\lambda_k$ cannot vanish for $k\geq K$.

Regarding the bounds $Y$, $Z_1$ and $Z_2$ needed to apply Theorem~\ref{th:NK}, $Y$ can again be computed by simply evaluating $A\F(\bu)$ (with interval arithmetic). For $Z_1$, we again introduce $B=I-AD\F(\bu)$ and separate its norm using Lemma~\ref{lem:norm_op_split}
\begin{align*}
\left\Vert B \right\Vert_\X = \max\left(\sup_{\substack{X \in \Pi_{3K-2} \X \\ X\neq 0}} \frac{\left\Vert BX \right\Vert_\X}{\left\Vert X \right\Vert_\X},\ \sup_{\substack{X \in \left(I-\Pi_{3K-2}\right) \X \\ X\neq 0}} \frac{\left\Vert BX \right\Vert_\X}{\left\Vert X \right\Vert_\X}\right).
\end{align*}
As in Section~\ref{sec:Z1_Lorenz}, we can again compute explicitly an upper-bound $Z_1^{finite}$ for the first supremum, and control the second one by
\begin{align*}
Z_1^{tail} = \frac{\left\Vert \rho -3\beta\, \bu\circledast\bu \right\Vert_\X}{\min\limits_{k\geq K} \lambda_k}.
\end{align*}
Finally, we can take $Z_2 = 6\vert\beta\vert \left\Vert A\right\Vert_\X \left(\left\Vert \bu\right\Vert_\X + r^*\right)$ for the last bound.

\subsection{Results in the 1-parameter case}

Here is an example of the type of results that can be obtained with this approach.
\begin{theorem}
\label{th:SH1}
Consider the 1D Swift--Hohenberg equation~\eqref{eq:SH} with $\beta=1$, $L=2\pi$, $\brho = 2.9$ and $\delta = 1.6$ in~\eqref{eq:norm_rho}, and the branch of approximate steady states $\bu$ represented in blue in Figure~\ref{fig:branches}, whose precise description in terms of Fourier$\times$Chebyshev coefficients can be downloaded at~\cite{BreGit}. Take also $\nu=\eta=1$.

With the notations introduced in Section~\ref{sec:SH_setup}, there exists a zero $u^*$ of $\F$ in $\X$ such that $\left\Vert \bu - u^*\right\Vert_\X \leq r_{min} = 3.8\times 10^{-4}$, and which is unique among all $u$ in $\X$ such that $\left\Vert \bu - u^*\right\Vert_\X \leq r_{max} = 6.5\times 10^{-3}$. This $u^*$ corresponds to an isolated branch of steady states of~\eqref{eq:SH} with $\beta=1$ and $L=2\pi$, for $\rho$ in $[1.3,4.5]$.
\end{theorem}

\begin{figure}[h!]
\includegraphics[width=0.495\linewidth]{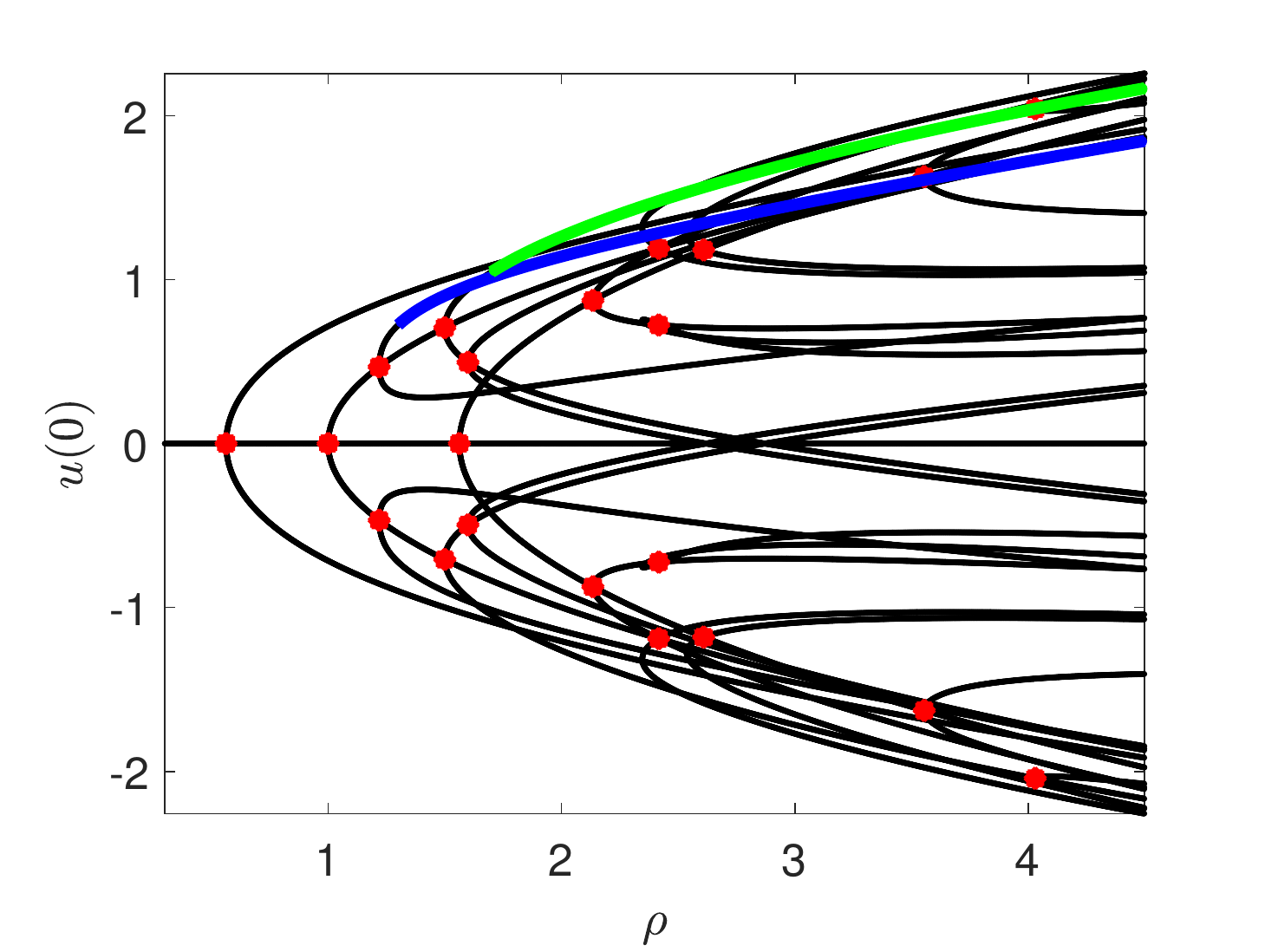} \hfill
\includegraphics[width=0.495\linewidth]{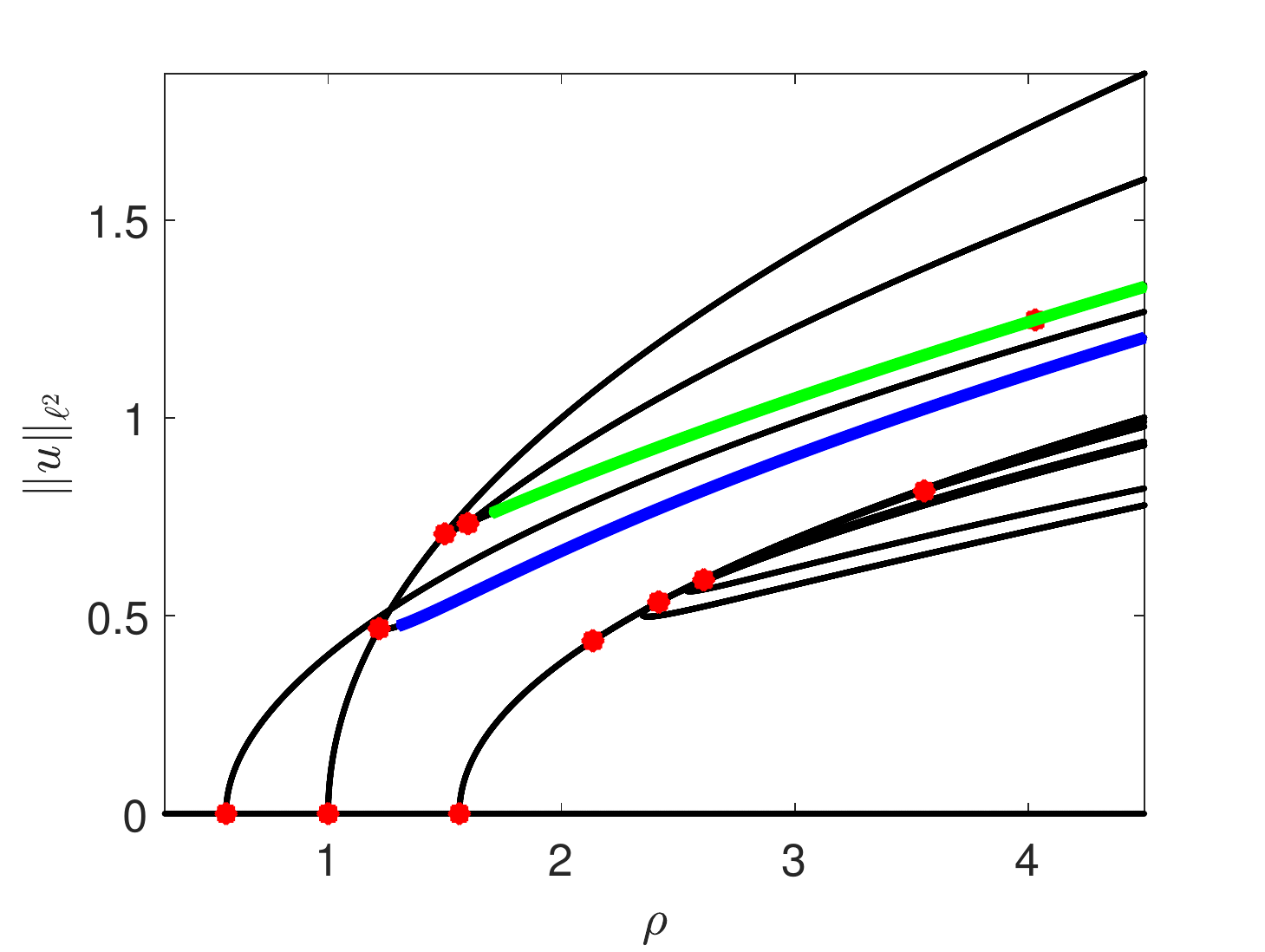}
\includegraphics[width=0.495\linewidth]{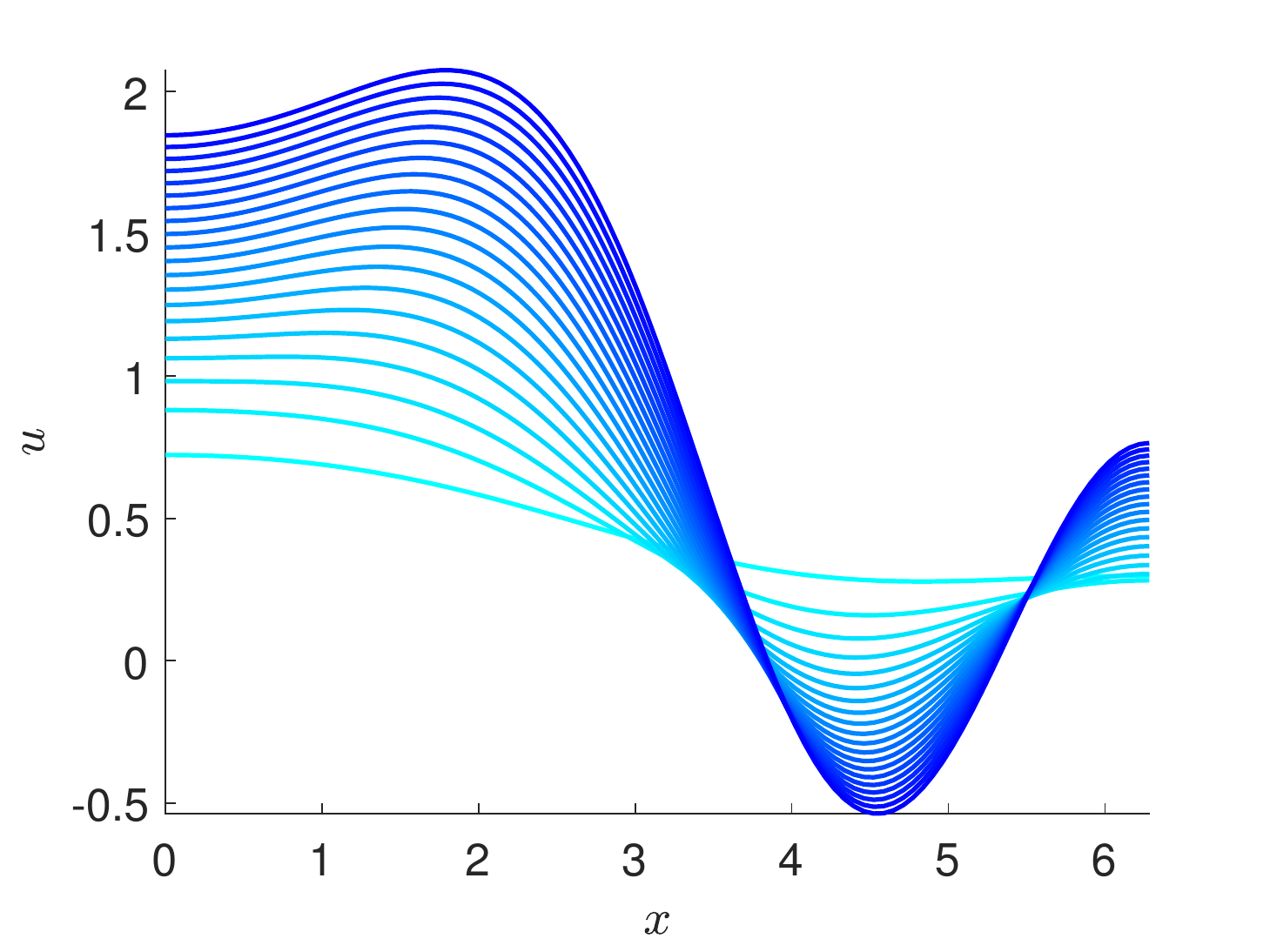} \hfill
\includegraphics[width=0.495\linewidth]{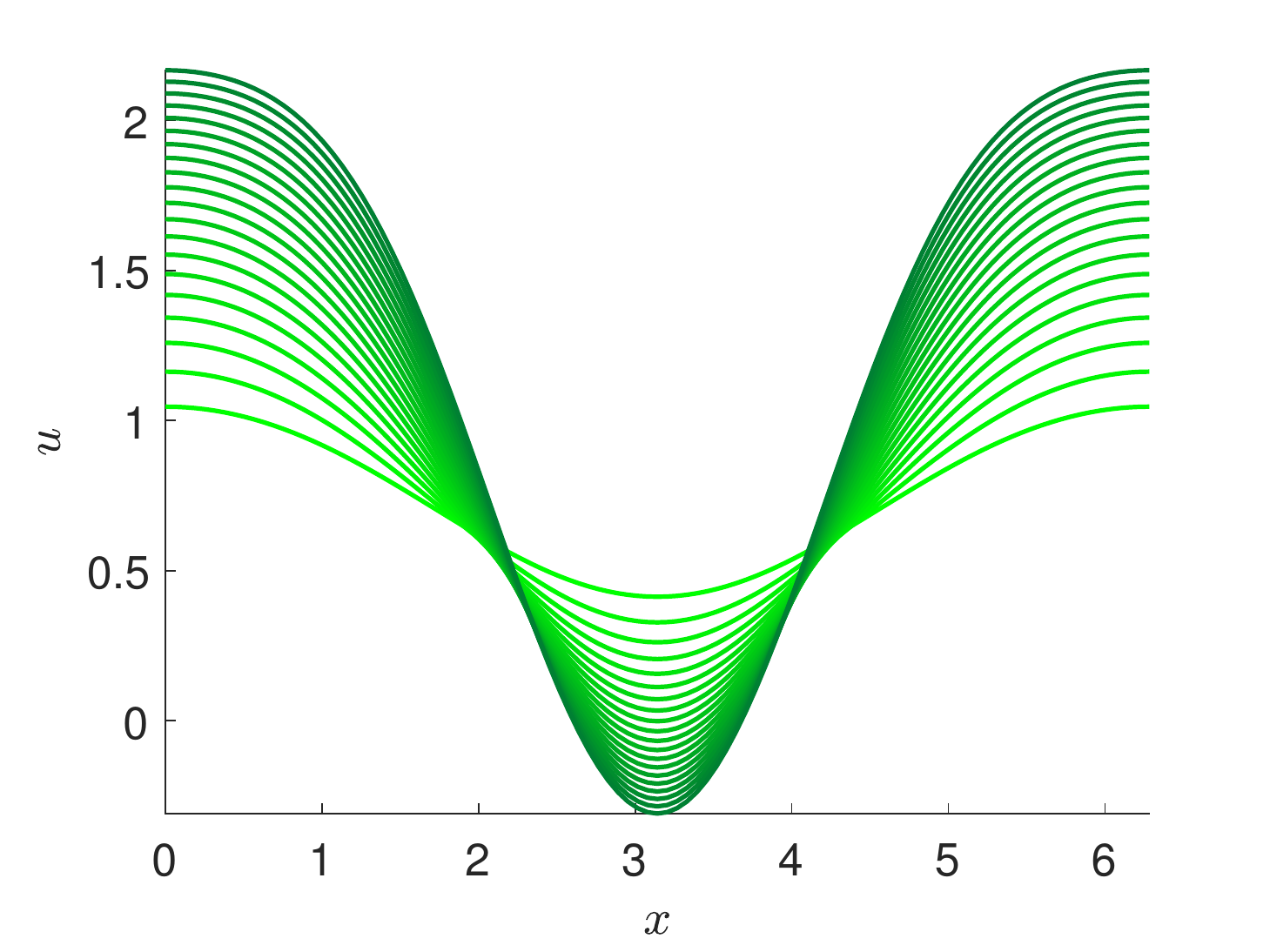}
\caption{Two specific portions of branches, in blue and green on the bifurcation diagrams at the top, for which several solutions along the branch are represented at the bottom (left for the blue branch and right for the green one). Regarding the solutions at the bottom, the lighter the color the smaller the corresponding value of $\rho$. The whole blue branch has been validated in Theorem~\ref{th:SH1}. The validation of green branch failed, which is coherent with the fact that numerics suggest the presence of a bifurcation crossing that branch.}
\label{fig:branches}
\end{figure}

\begin{proof}
We again evaluate the bounds $Y$, $Z_1$ and $Z_2$, obtained in Section~\ref{sec:SH_setup}, check that assumptions~\eqref{eq:cond} are satisfied, and apply Theorem~\ref{th:NK}, which yields the existence and uniqueness statement for a zero $u^*$ of $\F$ near $\bu$. The computational parts of the proof, namely the computation of the finite part $A_K$ of $A$ and the evaluation of the bounds, can be reproduced using \texttt{script\_SwiftHohenberg.m} available at~\cite{BreGit} (with Intlab~\cite{Rum99} for the required interval arithmetic computations). 

It remains to be proven that the branch of steady states corresponding to $u^*$ is isolated. This is essentially due to the fact that the $\ell^1_\eta$ norm controls the $\C^0$ norm (Lemma~\ref{lem:gPC_convo}), hence we can apply Theorem~\ref{th:NK} uniformly in $p$. To be more precise, for any $p$ in $[-1,1]$ and $u=\left(u_{k,n}\right)$ in $\X = \ell^1_\nu\left(\Z,\ell^1_\eta(\N,\R)\right)$, we can consider $u(p)=\left(u_{k}(p)\right)$ which now belongs to $\ell^1_\nu\left(\Z,\R\right)$, with $u_k(p)$ as in~\eqref{eq:up}. Thanks to Lemma~\ref{lem:gPC_convo}, we have that, for any $p$ in $[-1,1]$:
\begin{align}
\label{eq:C0_ell1}
\left\Vert u(p) \right\Vert_{\ell^1_\nu(\Z,\R)} \leq \left\Vert u \right\Vert_{\ell^1_\nu(\Z,\ell^1_\eta(\N,\R))}.
\end{align}
We also consider the map $\F_p$, which is defined as $\F$ but with $p$ fixed, and only acts on elements of $\ell^1_\nu\left(\Z,\R\right)$. Similarly, we recall that $A$ can be represented as an \emph{infinite matrix} $\left(A_{k,l}\right)_{k,l\in\Z}$ of operators on $\ell^1_\eta(\N,\C)$, and that each $A_{k,l}$ is in fact a multiplication operator on $\ell^1_\eta(\N,\R)$, represented by an element $a_{k,l}$ in $\ell^1_\eta(\N,\R)$. Hence we can consider $A(p) = \left(a_{k,l}(p)\right)_{k,l\in\Z}$, which now acts on $\ell^1_\nu(\Z,\R)$, where 
\begin{align*}
a_{k,l}(p) = \sum_{n\in\N} \left(a_{k,l}\right)_n \phi_n(p).
\end{align*}
Since the generalized convolution product of coefficients corresponds to the pointwise product of functions, we have that $A(p) \F_p(\bar u(p)) = \left(A\F(\bu)\right)(p)$, and by~\eqref{eq:C0_ell1}
\begin{align*}
\left\Vert A(p)\F_p(\bu(p)) \right\Vert_{\ell^1_\nu(\Z,\R)} \leq \left\Vert A\F(\bu) \right\Vert_{\X} \leq Y,
\end{align*}
for all $p$ in $[-1,1]$. Similarly,
\begin{align*}
\left\Vert I_{\ell^1_\nu(\Z,\R)} - A(p)D\F_p(\bu(p)) \right\Vert_{\ell^1_\nu(\Z,\R)} \leq \left\Vert I_\X - AD\F(\bu) \right\Vert_{\X} \leq Z_1,
\end{align*}
for all $p$ in $[-1,1]$. Indeed, for any linear operator $B=\left(B_{k,l}\right)_{k,l\in\Z}$ acting on $\X$, where each $B_{k,l}$ is a multiplication operator on $\ell^1_\eta(\N,\R)$ represented by an element $b_{k,l}$ in $\ell^1_\eta(\N,\R)$, we have (see Lemma~\ref{lem:norm_op})
\begin{align*}
\left\Vert B(p) \right\Vert_{\ell^1_\nu(\Z,\R)}=  \left\Vert \left(b_{k,l}(p)\right)_{k,l\in\Z} \right\Vert_{\ell^1_\nu}  \leq  \left\Vert \left(\left\Vert b_{k,l}\right\Vert_{\ell^1_\eta(\N,\R)}\right)_{k,l\in\Z} \right\Vert_{\ell^1_\nu} = \left\Vert B \right\Vert_{\X}.
\end{align*}
Finally, for any $p$ in $[-1,1]$ and any $v$ in $\B_{\ell^1_\nu(\Z,\R)}(\bu(p),r^*)$, we get
\begin{align*}
\left\Vert A(p)\left(D\F_p(v) - D\F(\bu(p))\right) \right\Vert_{\ell^1_\nu(\Z,\R)} &\leq  \left\Vert A(p) \right\Vert_{\ell^1_\nu(\Z,\R)} \left\Vert D\F_p(v) - D\F(\bu(p))\right\Vert_{\ell^1_\nu(\Z,\R)} \\
&\leq \left\Vert A(p) \right\Vert_{\ell^1_\nu(\Z,\R)} 6 \vert\beta\vert \left(\left\Vert \bu(p)\right\Vert_{\ell^1_\nu(\Z,\R)} + r^*\right) \\
&\leq Z_2.
\end{align*}
Hence, for each $p$ in $[-1,1]$ we can apply Theorem~\ref{th:NK} to the map $\F_p$ and the approximate solution $\bu(p)$, which proves that the steady state $u^*(p)$ of~\eqref{eq:SH} is locally unique, and in particular there cannot be a another branch of steady states of~\eqref{eq:SH} bifurcating from $u^*$.
\end{proof}

\begin{remark}
We used a Chebyshev expansion in $p$ in Theorem~\ref{th:SH1} because we expect it to be the most efficient choice to represent the branch of solutions. Indeed, while we could for instance have gotten a similar result with a Taylor expansion, we would have needed to take at least $N=47$ for assumption~\eqref{eq:cond2} to be satisfied, and $N=72$ if we wanted to get an error estimate $r_{min}$ which is as small as in Theorem~\ref{th:SH1}, whereas the current proof with a Chebyshev expansion uses only $N=15$.
\end{remark}

A remarkable part of Theorem~\ref{th:SH1} is that it guarantees that the portion of the branch that is validated is isolated, i.e. we have a proof that there is no other branch of steady states connected to this part. On the other hand, this means that we cannot hope to validate a part of a branch that goes through a bifurcation. Indeed, if we try to validate the branch of steady states represented in green in Figure~\ref{fig:branches}, the proof fails because $Z_1$ (in fact $Z_1^{finite}$) remains larger than $1$, no matter how large we take $K$ and $N$. This does not prove, but strongly suggests, that there is indeed a bifurcation on this part of the branch. Computer-assisted proofs of the existence of bifurcations are possibles, but require more work, see for instance~\cite{AriGazKoc21,AriKoc10,LesSanWan17,BerLesQue21} and the references therein. If the parameter $\rho$ is modeled by a random variable, one may want to try and quantify how these possible bifurcations impact the behavior of the system, which is for instance discussed in the recent work~\cite{KueLux21}.

\subsection{Extension to multi-parameter validated continuation}

Let us now consider both $\rho$ and $\beta$ as varying parameters in~\eqref{eq:SH}, normalized as
\begin{align}
\label{eq:norm_rho_beta}
\rho = \brho + \delta_\rho p_1,\qquad \beta = \bbeta + \delta_\beta p_2,
\end{align}
where $\brho,\bbeta$ and $\delta_\rho,\delta_\beta\geq 0$ are given constants, and $p_1,p_2$ vary in $[-1,1]$. Away from bifurcation points, we expect to get a 2-dimensional manifold of steady states parametrized by $p=(p_1,p_2)$. Using a bi-variate gPC expansion, we can approximate and then rigorously validate such manifold of steady states. It is remarkable that this generalization from the 1-parameter case requires only very minor modifications, both in terms of the estimates and in terms of the code. The only other work we are aware of in which validated multi-parameter continuation is studied is~\cite{GamLesPug16}, in which the transition from the 1-parameter case requires a significant effort.

Starting back from~\eqref{eq:SH_Four}, we now consider a bi-variate expansion for each Fourier coefficient
\begin{align*}
u_k(p) = u_k(p_1,p_2) = \sum_{n\in\N^2} u_{k,n} \phi_n(p),
\end{align*}
where the new basis is simply obtained by taking the tensor product of two univariate bases:
\begin{align*}
\phi_n(p):= \phi^{(1)}_{n_1}(p_1)\, \phi^{(2)}_{n_2}(p_2) \qquad \forall~n=(n_1,n_2)\in\N^2.
\end{align*}
In the sequel we take the Chebyshev polynomials of the first kind for both $\phi^{(1)}_{n_1}$ and $\phi^{(2)}_{n_2}$, but all the bases mentioned up to now could be combined here. 
A generalized convolution product associated to such bi-variate expansion can be defined in a straightforward way from the generalized convolution products associated to each univariate basis, see e.g.~\cite[Appendix]{BreKue20}.

Up to changing the space to $\X = \ell^1_\nu\left(\Z,\ell^1_{\eta_1}\left(\N,\ell^1_{\eta_2}(\N,\R)\right)\right) \simeq \ell^1_\nu\left(\Z,\ell^1_\eta(\N^2,\R)\right)$, to taking $\beta = \bbeta + \delta_\beta p_2$ instead of $\beta$ constant in $\F$, and to replacing $\vert \beta\vert$ by $\left\Vert \beta\right\Vert_{\ell^1_{\eta_2}(\N,\R)}$ in the $Z_2$ estimate, we can use exactly the same setup as in Section~\ref{sec:SH_setup} to validate an approximate 2-dimensional manifold of steady states.

\begin{theorem}
\label{th:SH2}
Consider the 1D Swift--Hohenberg equation~\eqref{eq:SH} with $L=2\pi$ and the manifold of approximate steady states $\bu$ represented in Figure~\ref{fig:manifold_u0}, whose precise description in terms of Fourier$\times$gPC coefficients can be downloaded at~\cite{BreGit}.

There exists a zero $u^*$ of $\F$ in $\X$ such that $\left\Vert \bu - u^*\right\Vert_\X \leq r_{min} = 4.4\times 10^{-3}$, and which is unique among all $u$ in $\X$ such that $\left\Vert \bu - u^*\right\Vert_\X \leq r_{max} = 8\times 10^{-3}$. This $u^*$ corresponds to an isolated manifold of steady states of~\eqref{eq:SH} with $L=2\pi$, for $(\rho,\beta)$ in $[2,4]\times[0.25,1.75]$.
\end{theorem}
\begin{proof}
The proof again amounts to checking the assumptions of Theorem~\ref{th:NK}. The computational parts of the proof, namely the computation of the finite part $A_K$ of $A$ and the evaluation of the bounds, can be reproduced using \texttt{script\_SwiftHohenberg\_2para.m} available at~\cite{BreGit} (with Intlab~\cite{Rum99} for the required interval arithmetic computations). 
\end{proof}

\begin{figure}[h!]
\includegraphics[width=0.6\linewidth]{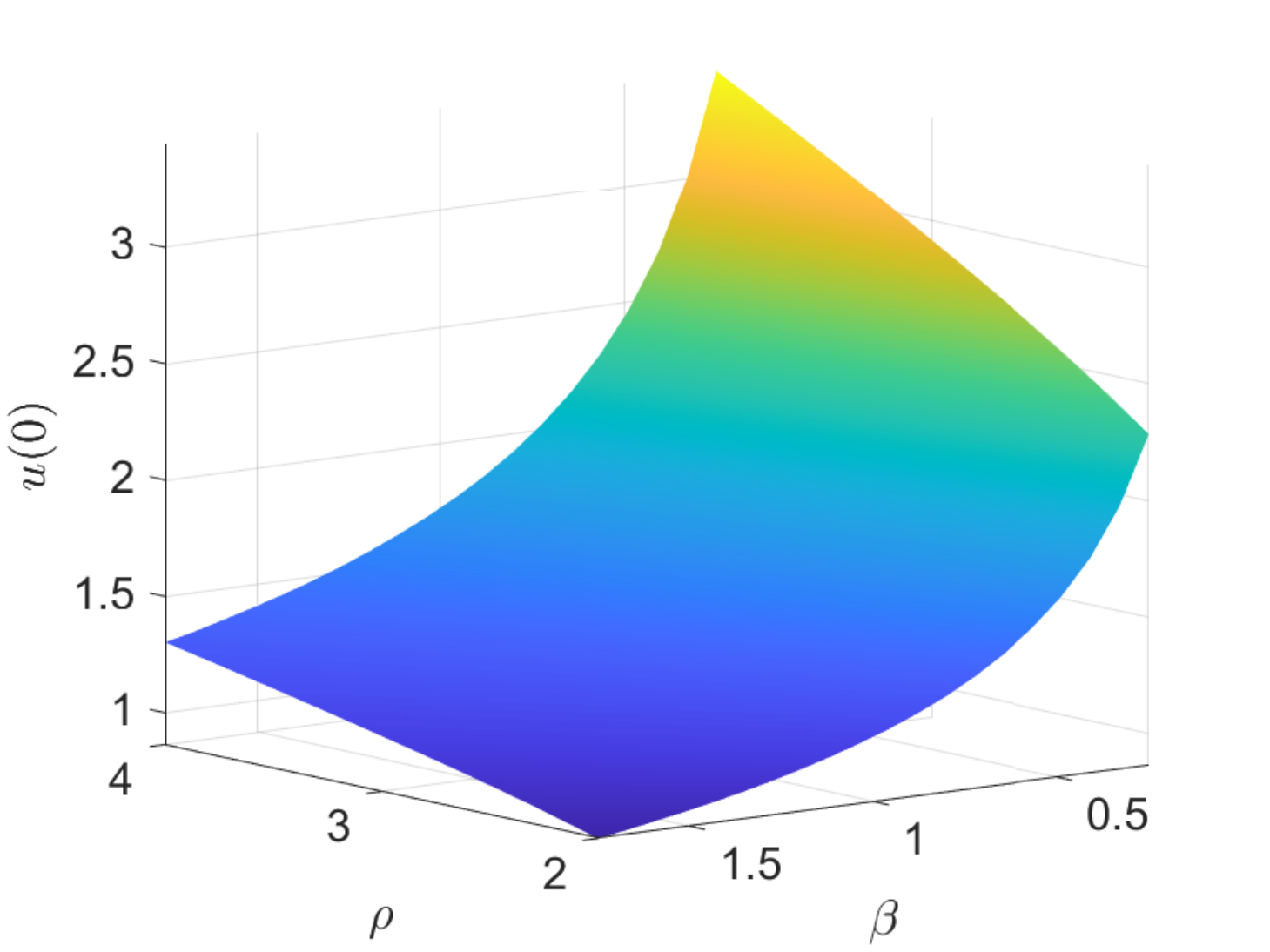}
\caption{A validated manifold of steady states of~\eqref{eq:SH}, for $L=2\pi$, and $(\rho,\beta)$ in $[2,4]\times[0.25,1.75]$, represented using $u(0)$ as a projection. The approximate manifold of steady states $\bu$ is composed of $K=20$ Fourier modes, $N_1 = 5$ Chebyshev modes in $\rho$ and $N_2 = 10$ Chebyshev modes in $\beta$.}
\label{fig:manifold_u0}
\end{figure}

\section{Conclusion}
\label{sec:conclusion}

In this work, we introduced a new methodology to obtain fully rigorous a posteriori error bounds for several types of gPC expansions (Legendre, Chebyshev of the first and the second kind, and Gegenbauer expansions). We showcased via several examples that this strategy can be used in the context of random invariant sets generated by random ODEs or PDEs, allowing to get a very precise and certified description of random periodic orbits of ODEs and of random steady states of parabolic PDEs.

These techniques can also be seen through the lens of rigorous/validated numerics, and in this context they provide a new way of rigorously computing curves or higher-dimensional manifolds of solutions in parameter-dependent systems, generalizing an approach introduced recently in~\cite{AriGazKoc21}. It is remarkable that the memory requirements associated to this approach can be made to scale linearly with the dimension of the gPC projection (see Remark~\ref{rem:Lorenz_memory}).

We finish by mentioning possible generalizations but also current limitations and open questions related to this work that we believe to be of interest.

\begin{itemize}
\item We only considered ODEs or PDEs with polynomial nonlinearities, which is a particularly convenient framework to work in with spectral techniques. Yet, some non-polynomial nonlinearities can be handled in a similar way, making use of ideas from automatic differentiation, see e.g.~\cite{LesMirRan16}.
\item While we only studied random steady states and random periodic orbits in this work, the proposed approach generalizes in a straightforward way to rigorously compute other types of random invariant sets, as soon as we already have the tools to rigorously compute them in the deterministic case, which is for instance the case for invariant manifolds or connecting orbits.
\item We restricted our attention to random parameters having somewhat classical distributions (namely uniform distributions or at least symmetric beta distributions). For more exotic distributions, in particular distributions that are obtained from data and have no analytic expression, one of the main difficulty with our approach is that we require an explicit knowledge of the linearization coefficients. We believe that generalizing the techniques of this paper to a wider class of random parameters (maybe making use of a probability transform to recover a uniform distribution) would be of interest.
\item Even if we stick with classical distributions, for which the linearization coefficients are known analytically, our approach can currently only handle bounded random parameters, and in particular excludes Gaussian or exponential distributions. The main reason is that the corresponding orthogonal polynomials, namely Hermite and Laguerre polynomials, do not readily give rise to a discrete convolution structure like the one we could make use of in this work (Lemma~\ref{lem:gPC_convo}). Finding a way to rigorously compute gPC expansions with those bases, which occur very naturally in many problems, would also be of great interest. 
\item We conclude with a comment about the implementation. Because we wanted to handle several different expansions in a uniform way, we did not take advantage of the fact that for some expansions (namely Chebyshev and Taylor expansions), the corresponding convolutions can be very efficiently computed using FFT (or DCT) algorithms. If one wanted to focus solely on Chebyshev expansions, which we would for instance recommend if one is only interested in the deterministic parameter-continuation viewpoint, making use of the FFT could improve the performances of the code significantly, especially for higher dimensional problems.
\end{itemize}


\bibliographystyle{abbrv}
\bibliography{../../../../Bibfile/bibfile}

\end{document}